\documentclass{amsart}

\usepackage{amssymb}
\usepackage{amsbsy}
\usepackage{amscd}
\usepackage{amsmath}
\usepackage{amsthm}
\usepackage{amsxtra}
\usepackage{mathtools}
\usepackage[new]{old-arrows} 

\usepackage{tikz}
\usetikzlibrary{cd}

\usepackage{pict2e}

\usepackage[mathscr]{eucal}
\usepackage{mathrsfs}
\usepackage{upgreek}
\usepackage{wasysym}
\usepackage[all,cmtip,dvipdfmx]{xy}
\usepackage{xcolor}
\definecolor{rouge}{rgb}{0.85,0.1,.4}
\definecolor{bleu}{rgb}{0.1,0.2,0.9}
\definecolor{violet}{rgb}{0.7,0,0.8}

\usepackage[colorlinks=true,linkcolor=bleu,urlcolor=violet,citecolor=rouge]{hyperref}

\usepackage{lmodern}

\usepackage{tdsfrmath}		
\usepackage{systeme}		
\usepackage{siunitx}	
\usepackage[shortlabels]{enumitem} 
\usepackage{xspace}			
\usepackage{xifthen} 	

\usepackage{dsfont}
\usepackage{stmaryrd}

\usepackage{graphicx}		
\usepackage{csquotes}		
\usepackage{varioref}	
\usepackage{caption}
\usepackage{subcaption}		

\allowdisplaybreaks[1] 		

\newcommand{\mc}{\mathcal}
\newcommand{\mf}{\mathfrak}

\renewcommand{\leq}{\leqslant}
\renewcommand{\geq}{\geqslant}

\newcommand{\bra}{{\langle}}
\newcommand{\ket}{{\rangle}}
\newcommand{\lam}{\lambda}
\newcommand{\Lam}{\Lambda}
\renewcommand{\Re}{{\rm Re}}

\newcommand{\Vir}{{\rm Vir}}
\newcommand{\W}{\mc{W}}

\newcommand{\g}{\mf{g}}
\newcommand{\h}{\mf{h}}
\newcommand{\n}{\mf{n}}
\newcommand{\affg}{\widehat{\mf{g}}}
\newcommand{\affh}{\widehat{\mf{h}}}
\newcommand{\extg}{\widetilde{\mf{g}}}
\newcommand{\exth}{\widetilde{\mf{h}}}
\renewcommand{\sl}[1]{\mf{sl}_{#1}}
\newcommand{\so}[1]{\mf{so}_{#1}}
\renewcommand{\sp}[1]{\mf{sp}_{#1}}
\newcommand{\VA}{\ensuremath{V^{k}(\g)}}
\newcommand{\SVA}{\ensuremath{V_{k}(\g)}}

\newcommand{\WA}{\ensuremath{\W^k(\g,f)}}
\newcommand{\SWA}{\ensuremath{\W_k(\g,f)}}

\renewcommand{\O}{\mathbb{O}}
\newcommand{\Clifford}{\ensuremath{\widehat{Cl}}}

\newcommand{\Fock}{\mc{F}}
\newcommand{\ens}[2]{\ifthenelse{\equal{#2}{pp}}{#1_{>0}}{\ifthenelse{\equal{#2}{p}}{#1_{\geq0}}{\ifthenelse{\equal{#2}{nn}}{#1_{<0}}{\ifthenelse{\equal{#2}{n}}{#1_{\leq0}}{#1_{#2}}}}}}

\newcommand{\vac}{{|0\rangle}}

\newcommand{\indic}[1]{\delta_{#1}}
\newcommand{\norm}[1]{\lvert#1\rvert}
\newcommand{\IP}[2]{(#1|#2)}

\newcommand{\NO}[1]{\ensuremath{:#1:}}
\newcommand{\OPE}[3]{\ensuremath{\frac{#2}{#3(z-w)^{#1}}}}
\renewcommand{\gcd}[2]{\ensuremath{(#1,#2)}}
\newcommand{\+}{\mathop{\oplus}}
\renewcommand{\*}{\otimes}
\newcommand{\Exterior}{\mathchoice{{\textstyle\bigwedge}}%
	{{\bigwedge}}%
	{{\textstyle\wedge}}%
	{{\scriptstyle\wedge}}}

\DeclareMathOperator{\Ima}{im}

\DeclareMathOperator{\id}{Id}

\DeclareMathOperator{\tr}{tr}

\DeclareMathOperator{\ch}{ch}
\DeclareMathOperator{\Zhu}{{\rm Zhu}}

\DeclareMathOperator{\Ext}{Ext}
\DeclareMathOperator{\Mod}{Mod}
\DeclareMathOperator{\killing}{\kappa_{\g}}

\DeclareMathOperator{\e}{e}

\newcommand{\cprime}{$'$}

\theoremstyle{theorem}
\newtheorem{Theorem}{Theorem}[section]

\theoremstyle{nonumberplain}
\newtheorem{MainTheoremNB}{Main~Theorem}

\theoremstyle{plain}
\newtheorem{Proposition}[Theorem]{Proposition}

\newtheorem{Lemma}[Theorem]{Lemma}

\newtheorem{Corollary}[Theorem]{Corollary}

\newtheorem{Conjecture}[Theorem]{Conjecture}

\theoremstyle{remark}

\newtheorem{Example}[Theorem]{Example}

\newtheorem{Remark}[Theorem]{Remark}

\usepackage{hyperref}

\newcommand{\charge}{c_k}
\newcommand{\vectop}{| \xi,\chi \rangle}
\newcommand{\vacij}[2]{|\xi^{(#2)}_{#1},\chi^{(#2)}_{#1}\rangle}
\renewcommand{\top}[1]{#1_{\text{top}}}
\newcommand{\Lxichi}{L(\xi,\chi)}
\newcommand{\Lxichiprime}{L(\xi',\chi')}
\newcommand{\Lxichiij}[2]{L(\xi^{(#2)}_{#1},\chi^{(#2)}_{#1})}

\title[Rationality of the exceptional $\W$-algebras $\W_k(\sp{4},f_{subreg})$]{\small    
Rationality of the exceptional $\W$-algebras $\W_k(\sp{4},f_{subreg})$ 
associated with subregular nilpotent elements of $\sp{4}$}
\author{Justine Fasquel}
\address{Univ. Lille, CNRS, UMR 8524 - Laboratoire Paul Painleve, F-59000 Lille, France}
\email{justine.fasquel@univ-lille.fr}

\begin{document}
\maketitle

\begin{abstract}
We prove the rationality of the exceptional $\W$-algebras $\SWA$ 
associated with the simple Lie algebra $\g=\sp{4}$ and a subregular nilpotent element $f=f_{subreg}$ 
of $\sp{4}$, proving a new particular case of a conjecture of Kac-Wakimoto. Moreover, we describe the simple $\SWA$-modules and compute their characters. We also explicit the nontrivial action of the component group on the set of simple $\SWA$-modules.
\end{abstract}

\section{Introduction}
Let $\g$ be a finite dimensional simple Lie algebra, $f$ a nilpotent element of $\g$ and $k\in\C$ a complex number. 
The universal affine $\W$-algebra $\WA$ associated with $(\g,f)$ is a certain vertex algebra obtained from the quantized Drinfeld-Sokolov reduction of the universal affine vertex algebra $\VA$. 
The $\W$-algebras can be regarded as affinizations of {\em finite $\W$-algebras} (introduced by Premet \cite{Pre02}), and can also be considered as generalizations of affine vertex algebras and Virasoro vertex algebras. 
The construction of $\W$-algebras was firstly introduced by Feigin and Frenkel \cite{FF90} for $f$ a principal nilpotent element, and was extended for general nilpotent elements by Kac, Roan and Wakimoto \cite{KRW03}.
The theory of $\W$-algebras is related with integrable systems \cite{Kac17}, the two-dimensional conformal field theory, the geometric Langlands program \cite{Gai16,Fre07,AF19}, and the 4d/2d duality \cite{Ara18,BLL15,BPR15,SXY17} in physics. 
 
The nicest (conformal) vertex algebras are those which are both rational and lisse. The rationality means the completely reducibility of $\Z_{\geq 0}$-graded modules. The lisse condition is equivalent to the fact that the {\em associated variety} has dimension~0.  
If a vertex algebra $V$ is rational and lisse, then it gives rise to a rational conformal field theory. 
The rationality condition implies that $V$ has finitely many simple $\Z_{\geq 0}$-graded modules and that the graded components of each of these $\Z_{\geq 0}$-graded modules are finite dimensional \cite{DLM98}.
In fact lisse vertex algebras also verify this property \cite{ABD04,MP95,Zhu96}.
It is actually conjectured by Zhu \cite{Zhu96} that rational vertex algebras must be lisse (this conjecture is still open). 

It is known \cite{FZ92} that the simple quotient $L_k(\g)$ of $\VA$ is rational (and lisse) if and only if it is integrable as a representation of $\affg$, that is, $k \in \Z_{\geq 0}$, where $\affg$ is the affine Kac-Moody algebra associated with $\g$. 
It is also known \cite{Wan93} that the simple quotient $\Vir_c$ of the Virasoro vertex algebra $\Vir^c$ is rational (and lisse) if and only if $c=c(p,q)$, where \[c(p,q)= 1 - \frac{6(p -q)^2}{pq}, \qquad p, q > 1, \; \gcd{p}{q}= 1.\]
More difficult is the study of the rationality of the simple quotient $\SWA$ of $\WA$. 

It was conjectured by Frenkel, Kac and Wakimoto \cite{FKW92}, and proved by Arakawa \cite{Ara15b}, that if $k=-h^\vee+p/q \in \mathbb{Q}$ 
is {\em non-degenerate admissible}\footnote{The level $k$ is {\em admissible} if $L_k(\g)$ is {\em admissible} as a representation of $\widehat{\g}$,  cf.~Sect.~\ref{sec:affine_vertex_algebras} for a precise definition. 
It is {\em non-degenerate} if the associated variety of 
$L_k(\g)$ is the whole nilpotent cone of $\g$.} for $\g$ and if $f=f_{princ}$ is a principal nilpotent element, then $\W_k(\g,f_{princ})$ is rational.  
Here, $h^\vee$ is the dual Coxeter number of $\g$.
More generally, Kac and Wakimoto \cite{KW08} conjectured that $\SWA$ is rational whenever $k=-h^\vee+p/q$ is admissible and $(f,q)$ is an \emph{exceptional pair} \cite{KW08}. 

It was shown by Arakawa \cite{Ara15a} that, for $k=-h^\vee+p/q$ an admissible level, the associated variety of the simple affine vertex algebra $L_{k}(\g)$ is precisely the closure of some nilpotent orbit $\O_q$ which only depends on the denominator $q$.
It is also proved in \cite{Ara15a} that if $f \in \O_q$ then $\SWA$ is lisse. 
Following \cite{AvE19}, we extend the notion of exceptional pair and say that the pair $(f,q)$ is {\em exceptional}\footnote{The notion of exceptional pair coincides with  \cite{EKV09} if $f$ is of {\em Levi type}, and with \cite{KW08}  if moreover $(q,r^\vee)=1$, where $r^\vee$ is the lacity of $\g$.} if $f \in \O_q$, with $q \in\Z$, $q>1$. 
Thus, for $k=-h^\vee+p/q$ an admissible level, the pair $(f,q)$ is exceptional only if $\SWA$ is lisse. 
It was conjectured in \cite{Ara15a} that $\SWA$ is rational whenever $k$ is admissible and $(f,q)$ is an exceptional pair, now in the broader sense. 

Arakawa and van Ekeren recently gave strong evidences of this conjecture by showing that the exceptional $\W$-algebra $\SWA$ is rational for a large class of exceptional pairs $(f,q)$. 
In particular, they proved the conjecture for exceptional $\W$-algebras in type $A$ (where all notions of exceptional pairs coincide) and for exceptional $\W$-algebras associated with a subregular nilpotent element $f$ when $\g$ is simply-laced. 
A particular case of this result was previously established by Arakawa  \cite{Ara13} using different methods for the Bershadsky-Polyakov vertex algebra, that is, the simple $\W$-algebra associated with $\sl{3}$ (type $A_2$) and $f$ subregular (in this case, $f$ is also a minimal nilpotent element of $\sl{3}$).

In this article, we prove the rationality of the exceptional $\W$-algebra $\W_k(\g,f)$ associated with $\g=\sp{4} \cong\so{5}$ (type $C_2=B_2$) and $f=f_{subreg}$ a subregular nilpotent element of $\g$. 
Note that the subregular nilpotent orbit $\O_{subreg}$ of $\sp{4}$ is associated with the partition $(2^2)$ of $4$ and that $f_{subreg}$ is of {\em Levi type} which means that $f$ is a principal nilpotent element in a Levi subalgebra of $\sp{4}$ containing it 
(see Remark \ref{Rem:Levi_type}). 
It follows from \cite{Ara15a} that $\O_q=\O_{subreg}$ if and only if $q=3$ or $q=4$.  
The level $k$ is admissible for $\sp{4}$ with denominator $q=3$ or $q=4$ if either $k= -3+p/3$ with $\gcd{p}{3}=1$ and $p\geq 3$, or $k= -3+p/4$ with $\gcd{p}{2}=1$ and $p\geq 4$. 

Specifically, our main result is the following, proving a conjecture of Kac-Wakimoto and Arakawa for $\g=\sp{4}$ and $f=f_{subreg}$.  

\begin{MainTheoremNB}\label{Th:main}
Let $f=f_{subreg}$ be a subregular nilpotent element of $\g=\sp{4}$. 
Then the exceptional $\W$-algebra $\W_{-3+p/3}(\g,f)$, with $\gcd{p}{3}=1$, $p\geq 3$, 
and the exceptional  $\W$-algebra $\W_{-3+p/4}(\g,f)$, with $\gcd{p}{2}=1$, $p\geq 4$, 
are rational (and lisse).
Moreover, we have a complete classification of their simple modules. 
\end{MainTheoremNB}

The $\W$-algebras of the main theorem can be viewed as the ``easiest" exceptional $\W$-algebras which are not covered by Arakawa and van Ekeren works. 
It is also a natural analog to the Bershadsky-Polyakov vertex algebra for the type $C_2$. 
In fact, the $\W$-algebra $\W_{-3+p/2}(\sp{4},f_{min})$, with $\gcd{p}{2}=1$ and $p\geq 4$, associated with $\g=\sp{4}$ and $f=f_{min}$ a minimal nilpotent element of $\sp{4}$ is also a natural analog.  
The latter is a bit more difficult to study than the one of our main theorem because of the number of its generators, but we also plan to study its rationality.   
More generally, we plan to study the rationality of the exceptional minimal $\W$-algebra $\W_{-h^\vee+p/2}(\sp{n},f_{min})$, $n \geq 4$, with $h^\vee=n/2+1$.  
Note that the closure of the minimal nilpotent orbit of $\sp{n}$, for any $n\geq 4$, appears as associated variety of simple admissible affine vertex algebras $L_k(\sp{n})$, while it does not for $\so{n}$ for $n\geq 7$.

We now say a few words about the proof. 
Contrary to the minimal $\W$-algebras, operator product expansions (OPE) are not known for an arbitrary $\W$-algebra.
Our first step is to compute explicit generators of $\W^{k}(\sp{4},f_{subreg})$, $k \in \C$, and OPEs between them. 
We refer to Sect.~\ref{sec:BRST} for general facts about $\WA$, for $f$ {\em even}\footnote{It means that the Dynkin grading associated with $f$ is even.}, and to Sect.~\ref{sec:OPE} for the computation of OPEs in the case where $\g=\sp{4}$ and $f=f_{subreg}$. Notice that the explicit OPEs are not only useful for the rationality property of $\W_{k}(\sp{4},f_{subreg})$ and they can be used to obtain remarkable isomorphisms of vertex algebra\footnote{For instance we will show in a future work that $\W_{-1}(\sp{4},f_{subreg})\simeq M(1)$ and $\W_{-2}(\sp{4},f_{subreg})\simeq\Vir^{-2}$. We thanks Dražen Adamović for pointing out these results to us.}.
Next, following ideas of \cite{Ara13}, we use the twist action introduced in \cite{Li97} to obtain a finite set of possible simple $\W_{k}(\sp{4},f_{subreg})$-modules, with $k$ as in the theorem (see Sect.~\ref{sec:twist}). 
Unfortunately, contrary to the case where $f=f_{min}$ is minimal, the functor $H_{f}^{0}(?)$ so that $\WA=H_f^{0}(\VA)$ is not exact for an arbitrary nilpotent element $f$. 
So we cannot use directly the methods of \cite{Ara13} to conclude that our set is exactly the set of simple $\W_{k}(\sp{4},f_{subreg})$-modules, 
with $k$ as in the theorem. 
To encounter the difficulty, we exploit technics of \cite{AvE19} to show that all elements of our set are simple $\W_{k}(\sp{4},f_{subreg})$-modules, 
and that there is no nontrivial extensions between them (see Sect.~\ref{sec:proof}). 

Some standard facts on vertex algebras are summarized in Sect.~\ref{sec:vertex_algebras}, 
and the main notation on affine Kac-Moody algebras and corresponding affine vertex 
algebras can be found in Sect.~\ref{sec:affine_vertex_algebras}. 

As a by-product of our proof, we obtain an explicit description of the simple $\W_{k}(\sp{4},f_{subreg})$-modules.
In Sect.~\ref{sec:componentgroup} we show that the component group of the nilpotent orbit $\O_{subreg}$ acts non-trivially on the finite set of the simple $\W_{k}(\sp{4},f_{subreg})$-modules.
Moreover, one uses this description to compute the characters of simple modules (see Sect.~\ref{sec:characters}) and formulate a conjecture about the link between simple $\W_{k}(\sp{4},f_{subreg})$-modules and certain highest weight representations of $\widehat{\sp{4}}$ (see Conjecture~\ref{conjecture}). It is then possible to deduce the fusion rules corresponding to the simple $\W$-algebra. We hope to go back to this topic in a future work. 

\subsection*{Acknowledgments} 
The author is very grateful to her thesis advisors Anne Moreau and Tomoyuki Arakawa for suggesting the problem and for useful conversations and comments. She also thanks the members of the $\W$-algebra working group of the University of Lille.

The author acknowledges support from the Labex CEMPI (ANR-11-LABX-0007-01).

\section{Vertex algebras}
\label{sec:vertex_algebras}
A \emph{vertex algebra} is a complex vector space $V$ equipped with a distinguished vector $\vac \in V$ and a linear map 
\[V \to ({\rm End}\,V)[[z,z^{-1}]], \quad a \mapsto a(z)=\sum_{n\in \Z} a_{(n)} z^{-n-1}\]
satisfying the following axioms: 
\begin{itemize}
\item $a(z)b \in V((z))$ for all $a,b \in V$, 
\item (vacuum axiom) $\vac (z) = \id_V$ and $a(z) \vac = a+ zV[[z]]$ for all $a \in V$, 
\item (locality axiom) $(z-w)^N [a(z),b(w)]=0$ for a sufficiently large $N$ for all $a,b \in V$. 
\end{itemize}
The {\em normally ordered product} of two fields $a(z)=\sum_{n\in\Z}a_{(n)}z^{-n-1}$ and $b(z)=\sum_{n\in\Z}b_{(n)}z^{-n-1}$ of a vertex algebra $V$ is defined by
\begin{equation*}
\NO{a(z)b(z)}=a(z)_+b(z)+b(z)a(z)_-,
\end{equation*}
where $a(z)_+:=\sum_{n<0}a_{(n)}z^{-n-1}$ and $a(z)_-:=\sum_{n\geq0}a_{(n)}z^{-n-1}$. 
The vertex algebra $V$ is said to be {\em strongly generated} \cite{KRW03} by a family of fields $\{a^i(z)\}_{i \in I}$ if the space of fields of $V$ is spanned by normally ordered products of the fields $\{a^i(z)\}_{i \in I}$ and their derivatives.
This means that, as a vector space, $V$ is spanned by 
\begin{align}\label{eq:PBW}
a_{(-n_1)}^{i_1} \ldots a_{(-n_s)}^{i_s} \vac
\end{align}
with $s \geq 0$, $n_r \geq 1$ and $i_r \in I$. 
The structure of $V$ is completely determined by the OPEs among the $a^i (z)$'s, $i\in I$, or, equivalently, the Lie brackets in End$(V)$ among the $a_{(n)}^{i_r}$'s. 
If $V$ is strongly generated by the fields $\{a^i(z)\}_{i \in I}$, we call the set of monomials \eqref{eq:PBW}, where the sequence of pairs $(i_1,n_1),\ldots,(i_r,n_r)$ is decreasing in the lexicographical order, a {\em PBW basis} of $V$. 

A vertex algebra $V$ is called {\em conformal} if there exists a vector $\omega$ called the {\em conformal vector} such that
 $L(z)=\sum_{n\in \Z}L_nz^{-n-2}:=Y(\omega,z)$ satisfies that 
\begin{itemize}
\item[(a)] $[L_m,L_n]=(m-n)L_{m+n}+\frac{m^3-m}{12}\delta_{m+n,0}c_V$,
 where $c_V$ is some constant called the {\em central charge} of $V$,
\item[(b)] $L_0$ acts semisimply on $V$,
\item[(c)] $L_{-1}=T$ on $V$, where $T\colon V \to V$, $a \mapsto a_{(-2)}\vac$ is the translation operator. 
\end{itemize}
Here, $\indic{i,j}$ stands for the Kronecker symbol. 
For a conformal vertex algebra $V$ and a $V$-module $M$, we set $M_d=\{m\in M\colon L_0 m=dm\}$.
The $L_0$-eivenvalue of a nonzero $L_0$-eigenvector $m$ is called the {\em conformal weight}. 

If $V$ is conformal and $a \in V$ is homogeneous of conformal weight $\Delta_a$, we set $a_n=a_{(n+\Delta_a-1)}$ so that 
\begin{align}\label{eq:conformal_fileds}
a(z) = \sum_{n\in\Z} a_n z^{-n-\Delta_a},
\end{align}
which is more standard notation in physics.
 
Let $M$ be a module over a conformal vertex algebra $V$ of central charge $c$.
The module $M$ is called a {\em positive energy representation} if $L_0$ acts semisimply with spectrum bounded below, that is,
\[M=\bigoplus\limits_{d \in \chi + \Z_{\geq 0}}M_{d},\]
with $M_{\chi} \not= 0$. 
Let $\top{M}$ be the {\em top degree component} $M_{\chi}$ of $M$.

By Zhu's theorem \cite{Zhu96}, the correspondence $M \mapsto \top{M}$ gives a bijection between the set of isomorphism classes of irreducible positive energy representations of $V$ and that of simple $\Zhu(V)$-modules, where $\Zhu(V)$ is the {\em Zhu algebra} of $V$ (see, for example, \cite{AM20} 
for more detail). 

With every vertex algebra $V$ one associates a Poisson algebra $R_V$, called the {\em Zhu $C_2$-algebra}, as follows (\cite{Zhu96}). 
Let $C_2(V)$ be the subspace of $V$ spanned by the elements $a_{(-2)}b$, where $a,b \in V$, and set $R_V =V/C_2(V)$. 
Then $R_V$ is naturally a Poisson algebra by
\[1= \overline{\vac}, \qquad \bar{a} .\bar{b} = \overline{a_{(-1)}b} \quad 
\text{ and }\quad \{\bar{a},\bar{b}\} =\overline{a_{(0)}b},\]
where $\bar{a}$ denotes the image of $a\in V$ in the quotient $R_V$. 

The {\em associated variety} \cite{Ara12} $X_V$ of a vertex algebra $V$ is the affine Poisson variety defined by 
\[X_V = {\rm Specm}\, R_V.\]
A vertex algebra $V$ called {\em lisse} if $\dim X_V=0$.
 
\section{Affine vertex algebras}\label{sec:affine_vertex_algebras}
Let $\g$ be a complex simple Lie algebra as in introduction.
Let $\g=\n_- \+ \h\+ \n_+$ be a triangular decomposition with a Cartan subalgebra $\h$, $\Delta$ the root system of $(\g,\h)$, $\Delta_{+}$ a set of positive roots for $\Delta$ and $\Pi=\{\alpha_1,\dots,\alpha_{\ell}\}$ the corresponding set of simple roots.
Let $\theta$ be the highest positive root. 
We also have $\Delta^\vee$ the set of coroots.
Let $P$ be the weight lattice, $Q$ the root lattice and $Q^\vee$ the coroot lattice. 

Let $\extg = \g[t,t^{-1}] \oplus \C K \+ \C D$  be the extended affine Kac-Moody algebra, with the commutation relations:  
\[[x t^m,y t^n] = [x,y] t^{m+n} + m \indic{m+n,0} (x|y) K, \quad [D,xt^n]=-nxt^n,\quad [K,\affg]=0,\]
for all $x,y \in\g$ and all $m,n\in\Z$, where $(~|~)=\displaystyle{\frac{1}{2 h^\vee}}\times\killing$ is the normalized invariant inner product of $\g$, $\killing$ is the Killing form of $\g$ and $x t^n$ stands for $x \otimes t^n$, for $x \in \g$, $n \in \Z$. 
Let $\extg=\widehat{\n}_-\+ \exth\+\widehat{\n}_+$ be the standard triangular decomposition, that is, $\exth = \h \oplus \C K\+ \C D $ is a Cartan subalgebra of $\extg$, $\widehat{\n}_+=\n_++t\g[t]$ and $\widehat{\n}_-=\n_-+t^{-1}\g[t^{-1}]$.

Let $\affg =[\extg,\extg]= \g[t,t^{-1}] \oplus \C K$, and let $\affh=\h\+ \C K\subset \affg$, so that $\affg=\widehat{\n}_-\+ \affh\+ \widehat{\n}_+$.
The Cartan subalgebra $\exth$ is equipped with a bilinear form extending that on $\h$ by
\begin{align*}
(K | D) = 1, \quad (\h | \C K \oplus \C D) = (K | K) = (D | D) = 0.
\end{align*}
We write $\Lambda_0$ and $\delta$ for the elements of $\exth^*$ orthogonal to $\h^*$ and dual to $K$ and $D$, respectively. 
We have the (real) root systems
\begin{align*}
 \widehat{\Delta}^{re} &= \{\alpha+n\delta  \mid n \in \Z, \alpha \in \Delta\}= \widehat{\Delta}^{re}_+\sqcup  (-\widehat{\Delta}^{re}_+),\\
 &\widehat{\Delta}^{re}_+=\{\alpha+n\delta\mid \alpha\in \Delta_+, \ n\geq 0\}\sqcup \{-\alpha+n\delta\mid \alpha\in \Delta_+, \ n> 0\},
\end{align*}
and the affine Weyl group $\widehat{W}$, generated by reflections $r_\alpha$ with $\alpha\in  \widehat{\Delta}^{re}$. 
For $\alpha \in \h^*$ the translation $t_\alpha : \exth^* \rightarrow \exth^*$ is defined by
\begin{align*}
t_\alpha(\lam) = \lam + \lam(K)\alpha - \left[(\alpha | \lam) + \frac{|\alpha|^2}{2} \lam(K) \right] \delta.
\end{align*}
For $\alpha \in Q^\vee$ we have $t_\alpha \in \widehat{W}$ and in fact $\widehat{W} \cong W \ltimes t_{Q^\vee}$, where $t_{Q^\vee}:=\{t_\alpha \colon \alpha \in Q^\vee\}$ and $W$ is the Weyl group of $\g$. 
The extended affine Weyl group, which is the group of isometries of $\widehat{\Delta}$, is $\widetilde{W} = W \ltimes t_{P}$, 
where $t_{P}:=\{t_\alpha \colon \alpha \in P\}$.

Let $\mc{O}_k$ be the category $\mc{O}$ of $\extg$ at level $k$ (\cite{Kac90}). 
The simple objects of $\mc{O}_k$ are the irreducible highest weight representations $L(\lam)$ for $\lam\in \exth^*$ with $\lam(K)=k$.
For a weight $\lam \in \affh^*$ the corresponding {\em integral root system} is
\begin{align*}
\widehat{\Delta}(\lam) = \{\alpha\in \widehat{\Delta}^{re} \mid \left<\lam, \alpha^\vee\right> \in \Z\},
\end{align*}
where $\alpha^{\vee}=2\alpha/(\alpha|\alpha)$ as usual and let $\widehat{W}(\lambda)$ the corresponding {\em integral Weyl group} generated by the reflections $r_\alpha$, $\alpha\in\widehat{\Delta}(\lambda)$.

A weight $\lam \in \affh^*$ is said to be \emph{admissible} (equivalently, we said that $L(\lam)$ is {\em admissible}) if 
\begin{itemize}
\item $\lam$ is regular dominant, that is, $\bra \lam+\hat{\rho},\alpha^{\vee}\ket>0$ for all $\alpha\in\widehat{\Delta}_+(\lam)=\widehat{\Delta}(\lam)\cap 
\widehat{\Delta}^{re}_+$,
\item $\mathbb{Q}\widehat{\Delta}^{re} =\mathbb{Q}\widehat{\Delta}(\lam)$.
\end{itemize}
Here $\hat{\rho}=\rho+h^{\vee}\Lam_0$, with $\rho=\sum_{\alpha\in \Delta_+}\alpha/2$. 
A complex number $k$ is said to be {\em admissible} if $\lam =k \Lam_0$ is an admissible weight. 

\begin{Proposition}[{\cite{KW89}}]
A complex number $k$ is admissible if and only if it is of the form
	\[k=-h^\vee+\frac{p}{q},\]
with $p,q\in \Z_{> 0}$, $\gcd{p}{q}=1$ and either $p\geq h^\vee$ if $\gcd{r^\vee}{q}=1$, or $p\geq h$ if $\gcd{r^\vee}{q} =r^\vee$. 
Here $r^\vee$ is the lacity of $\g$, that is,
$r^\vee=1$ if $\g$ has type $A_\ell,D_\ell,E_6,E_7,E_8$, 
$r^\vee=2$ if $\g$ has type $B_\ell,C_\ell,F_4$ and 
$r^\vee=3$ if $\g$ has type $G_2$.
\end{Proposition}

For $k$ an admissible number, let $Pr^k$ be the set of weights $\lambda\in\h^*$ such that $\hat{\lam}=\lam+k\Lambda_0$ is admissible and there exists $y\in\widetilde{W}$ such that $\widehat{\Delta}(\hat{\lambda})=y(\widehat{\Delta}(k\Lambda_0))$.
Notice that for $\lam\in Pr^k$, if $y\in\widetilde{W}$ satisfies $y(\widehat{\Delta}(k\Lambda_0)_+)\subset \widehat{\Delta}^{re}_+$ and $\widehat{\Delta}(\lambda)=y(\widehat{\Delta}(k\Lambda_0))$, then $\widehat{W}(\lambda)=y\widehat{W}(k\Lambda_0)y^{-1}$.
The weights of $Pr^k$ are said \emph{principal admissible} if $\gcd{r^\vee}{q}=1$ and \emph{coprincipal admissible} if $\gcd{r^\vee}{q}=r^\vee$.

\subsection{Affine vertex algebras}
Given any $k\in\C$, let
\begin{align*} 
\VA = U(\affg)\*_{U(\mf{g}[t]\+ \C K)}\C_k, 
\end{align*}
where $\C_k$ is  the one-dimensional representation of $\g[t] \+ \C K$ on which $\g[t]$ acts by 0 and $K$ acts as a multiplication by the scalar $k$.
There is a unique vertex algebra structure on $\VA$ such that $\vac$ is the image of $1\* 1$ in $\VA$ and 
\[x(z) :=(x_{(-1)}\vac)(z)= \sum_{n\in\Z} (xt^n) z^{-n-1}\]
for all $x\in \g$, where we regard $\g$ as a subspace of $V$ through the embedding $x \in \g \hookrightarrow x_{(-1)}\vac \in \VA$. 
The vertex algebra $\VA$ is called the {\em universal affine vertex algebra} associated with $\g$ at level $k$.

The vertex algebra $\VA$ is conformal by Sugawara construction provided that $k$ is non-critical, that is, $k\ne -h^{\vee}$,
with central charge
\[c_{\VA}=  \dfrac{k  \dim \g }{k + h^\vee}.\]
We write \[L^\g(z)=\sum_{n\in \Z}L_nz^{-n-2}=Y(\omega_{\g},z),\]
where $\omega_\g$ is the Sugawara conformal vector.

A $\VA$-module is the same as a smooth $\affg$-module of level $k$.
For a non-critical level $k$, we consider a $\VA$-module $M$ as a $\extg$-module by letting $D$ act as the semisimplification of $-L_0$.
Let $L_k(\g)$ be the unique simple graded quotient of $\VA$. 
Then $L_k(\g) \cong L(k \Lambda_0)$ as a representation of $\affg$. 
 
For any graded quotient $V$ of $\VA$, we have $R_V=V/t^{-2}\g[t^{-1}] V$. 
In particular, $R_{\VA} \cong \mathbb{C}[ \g^* ]$ and, hence, $X_{\VA}=\g^*$. 
Furthermore, $X_{L_k(\g)}$ is a subvariety of $\g^* \cong \g$, 
which is $G$-invariant and conic with $G$ the adjoint group of $\g$. 
The associated variety $X_{L_k(\g)}$ is difficult to compute in general. 
However, it is known that $X_{L_k(\g)}=\{0\}$ 
if and only if $L_k(\g)\cong L(k \Lambda_0)$ is an integral representation of $\affg$, that is, $k \in \Z_{\geq 0}$. 
Furthermore, we have the following result: 

\begin{Proposition}[{\cite{Ara15a}}]
\label{Pro:associated_variety_admissible}
If $k=-h^\vee+p/q$ is an admissible level for $\g$, then $X_{L_k(\g)}$ 
is the closure of some nilpotent orbit $\O_q$ which only depends on $q$. 
\end{Proposition}

The nilpotent orbit $\O_q$ is explicitly described in \cite[Tables 2--10]{Ara15a}. 

\section{The BRST reduction and $\W$-algebras}\label{sec:BRST}
Let $f$ be a nilpotent element of $\g$ that we embed into an $\sl{2}$-triple $(e,h,f)$ through the Jacobson-Morosov theorem:
\begin{equation*}
[h,e]=2e,\quad [h,f]=-2f,\quad [e,f]=h.
\end{equation*}
The semisimple element $x_0:=h/2$ induces an $\frac{1}{2}\Z$-gradation on $\g$,
\[\g = \bigoplus_{j \in \frac{1}{2}\Z} \g_j,\]
where $\g_j=\{y \in\g \colon [x_0,y ] = j y \}$. 
Set $\g_{\geq 0}:=\bigoplus_{j\geq0}\g_j$ and $\g_{>0}:=\bigoplus_{j>0}\g_j$. 
We define similarly $\g_{\leq 0}$ and $\g_{< 0}$.
In this note, we assume that 
\[\g_j= 0 \quad \text{ for } \quad j\not\in \Z,\]
which is sufficient for our purpose.  
In other words, we assume that $f$ is an {\em even} nilpotent element (which is the case of subregular nilpotent elements of $\mf{sp}_4$). 
Choose a basis $\{e_\alpha\}_{\alpha\in S_j}$ of each $\g_j$. 
One can assume that each $e_\alpha$ is a root vector provided that $j \not= 0$, and that for $j=0$ either $e_\alpha$ is a root vector or $e_\alpha$ belongs to the Cartan subalgebra $\h$. 
Set $S = \sqcup_{j}S_j$, $S_+ = \sqcup_{j >0}S_j$ and let $\{e^\alpha\}_{\alpha \in S_+}$ be the dual basis in $\g_{>0}^*$ to $\{e_\alpha\}_{\alpha\in S_+}$. 

The {\em Clifford affinization} $\Clifford({\g}_{>0})$ of $\g_{>0}$ is the Clifford algebra associated with $\g_{>0}[t,t^{-1}] \oplus \g^*_{>0}[t,t^{-1}]$ 
and the symmetric bilinear form defined by 
\[\langle x t^m , \psi t^n \rangle = \delta_{m+n,0} \psi(x), \quad 
\langle x t^m , y t^n \rangle =  \langle \psi t^m , \phi t^n \rangle =0,\]
for $x,y \in \g_{>0}$, $\psi, \phi \in \g_{>0}^*$. 
We write $\varphi_{\alpha,m}$ for $e_\alpha t^m \in \Clifford({\g}_{>0})$ and $\varphi^{\alpha}_{m}$ for $e^\alpha t^m \in \Clifford({\g}_{>0})$, so that $\Clifford({\g}_{>0})$ is the associative superalgebra with 
\begin{itemize}
\item odd generators: $\varphi_{\alpha,m}$, $\varphi^{\alpha}_{n}$, $m,n\in \Z$, $\alpha \in S_+$, 
\item relations: $[\varphi_{\alpha,m}, \varphi_{\beta,n}]= [\varphi^{\alpha}_{m}, \varphi^{\beta}_{n}]=0$, 
$[\varphi_{\alpha,m}, \varphi^{\beta}_{n}]=\delta_{\alpha, \beta}\delta_{m+n, 0}$,
\end{itemize}
where the parity of $\varphi_{\alpha,m}$ and $\varphi^\alpha_n$ is reverse to $e_\alpha$. 
Since elements of $\g$ are purely even, this means that $\varphi_{\alpha,m}$ and $\varphi^{\alpha}_n$ are odd. 

Define the {\em charged fermion Fock space} associated with $\g_{>0}$ as
\begin{equation*}
\Fock(\g_{>0}):= \dfrac{\Clifford({\g}_{>0})}{\sum\limits_{m\geq 0\atop \alpha \in S_+}\Clifford({\g}_{>0}) \varphi_{\alpha,m}+\sum\limits_{k\geq 1\atop \alpha \in S_+} \Clifford({\g}_{>0}) \varphi^\alpha_k}
\cong \Exterior \big(\varphi_{\alpha,n}\big)_{\substack{n<0 \\ \alpha \in S_+}} 
\* \Exterior \big(\varphi^{\alpha}_m\big)_{\substack{m \leq 0\\ \alpha \in S_+}},
\end{equation*}
where $\Exterior (a_i)_{i\in I}$ denotes the exterior algebra with generators $a_i$, $i \in I$. 
It is an irreducible $\Clifford(\g_{>0})$-module, and as $\C$-vector spaces we have
\begin{equation*}
\Fock(\g_{>0}) \cong  \Exterior (\g_{>0}^\ast[t^{-1}])\* \Exterior(\g_{>0}[t^{-1}]t^{-1}).
\end{equation*}
There is a unique vertex (super)algebra structure on $\Fock(\g_{>0})$ such that the image of $1$ is the vacuum $\vac$  and 
\begin{align*}
& Y(\varphi_{\alpha,-1}|0\ket,z)=\varphi_\alpha(z):=\sum_{n\in \Z}\varphi_{\alpha,n} z^{-n-1}, 
\quad \alpha \in S_+,\\
&Y(\varphi^\alpha_{0}|0\ket,z)=\varphi^\alpha(z):=\sum_{n\in \Z}\varphi^\alpha_{n} z^{-n}, 
\quad  \alpha \in S_+.
\end{align*}
We call this vertex algebra the {\em free superfermion vertex algebra} associated with $\Clifford({\g}_{>0})$. 

The following construction of the $\W$-algebra associated with $\g$ and $f$ is due to Kac, Roan and Wakimoto \cite{KRW03,KW04,Ara05} (see also \cite{AM20} for a survey).
Let $k \in \C$ and set 
\[\mathscr{C}(\g,f,k)=\VA \otimes \Fock(\g_{>0}).\]
We now define a differential $d_{(0)}$ on $\mathscr{C}(\g,f,k)$. 
The vertex algebra $\Fock(\g_{>0})$ has the charge decomposition: 
\[\Fock(\g_{>0})= \bigoplus_m \Fock_m(\g_{>0}),\]
where ${\rm charge}\, \varphi_\alpha(z) = - {\rm charge}\, \varphi^\alpha(z) =1$ for $\alpha \in S_+$. 
Letting ${\rm charge}\,\VA=0$, this induces the charge decomposition: 
\[\mathscr{C}(\g,f,k)=\bigoplus_m \mathscr{C}_m, \qquad \mathscr{C}_m:=\mathscr{C}_m(\g,f,k).\]

Following \cite{KRW03}, we set 
\begin{align*}
	d(z) = & \sum_{\alpha\in S_+} 
	\NO{e_\alpha(z) \varphi^\alpha(z)} 
	- \frac{1}{2}\sum_{\alpha,\beta,\gamma\in S_+} 
	c_{\alpha,\beta}^\gamma 
	{  \varphi^\alpha(z)\varphi^\beta(z) \varphi_\gamma(z)} 
	& +\sum_{\alpha\in S_+} (f | e_\alpha) \varphi^\alpha(z),
\end{align*}
where $[e_\alpha,e_\beta]=\sum_{\gamma}c_{\alpha,\beta}^\gamma e_\gamma$. 
The field $d(z)$ does not depend on the choice of the basis. 
One has $[d(z),d(w)]=0$ and therefore $d_{(0)}^2=0$ since $d(z)$ is odd. 
Moreover, $[d_{(0)},\mathscr{C}_m]\subset\mathscr{C}_{m-1}$. 
Thus $(\mathscr{C}(\g,f,k),d_{(0)})$ is a $\Z$-graded homology complex. 
The zero-th homology of this complex is a vertex algebra, denoted by $\WA$: 
\[\WA := H^0(\mathscr{C}(\g,f,k),d_{(0)}),\]
that we briefly write 
$$\WA=H_f^0(\VA).$$
This construction is a particular case of BRST reduction which is usually referred to as the Drinfeld-Sokolov reduction of $\VA$. 
The vertex algebra $\W^k(\g,f)$ is called the {\em (universal) $\W$-algebra associated with $\g$ and $f$} at the level $k$.
Its simple graded quotient will be denoted by $\SWA$. 

Provided that $k\not= -h^\vee$, define the conformal field by:
\[L(z)=L^\g(z) +\frac{d}{dz} x(z) + L^{ch}(z),\]
where 
\[L^{ch}(z) = - \sum_{\alpha\in S_+} m_\alpha \NO{\varphi^\alpha \partial \varphi_\alpha} 
+  \sum_{\alpha\in S_+} (1-m_\alpha) \NO{\partial  \varphi^\alpha \varphi_\alpha}.\]
Here $m_\alpha$ is defined by letting $m_\alpha=j$ is $e_\alpha \in \g_j$. 
The central charge of $\WA$ is: 
\[c_{\WA}=\dim \g_0 -\dfrac{12}{k+h^\vee} | \rho - (k+h^\vee) x |^2.\]

It is known  \cite{DK06} that the associated variety of $\WA$ is the {\em Slodowy slice} $\mathscr{S}_f:=f +\g^{e}$, where $\g^{e}$ is the centralizer of $e$ in $\g$. 
Moreover by \cite{Ara15a}, the associated variety of $H_{f}^0(L_k(\g))$ is $X_{L_k(\g)} \cap \mathscr{S}_f$. 
Since $\W_k(\g,f)$ is a quotient of $H_{f}^0(L_k(\g))$, provided that $H_{f}^0(L_k(\g)) \not=0$, we deduce that $\W_k(\g,f)$ is lisse whenever $X_{L_k(\g)} =\overline{\mathbb{O}}$ and $f \in\mathbb{O}$, where $\mathbb{O}$ is some nilpotent orbit of $\g$. 
Indeed, we have $\overline{\mathbb{O}} \cap  \mathscr{S}_f=\{f\}$ if $f \in  \mathbb{O}$.
In particular, we get the following result,  useful  for our purpose.

\begin{Proposition}[{\cite{Ara15a}}]\label{Pro:associated_variety_W-algebra}
Assume that $k=-h^\vee+p/q$ is admissible for $\g$ and pick $ f\in \mathbb{O}_q$, with $\overline{\mathbb{O}_q}$ the associated variety of $L_k(\g)$ (cf.~Proposition \ref{Pro:associated_variety_admissible}). 
Then $\W_k(\g,f)$ is lisse. 
\end{Proposition}

\subsection{Generating fields of $\W^k(\g,f)$}\label{sec:fields}
Given $a \in \g_j$, 
introduce the following fields of $\mathscr{C}(\g,f,k)$ of conformal weight $1-j$:
\begin{equation*}
J^a(z)=a(z) + 
\sum_{\alpha,\beta\in S_+}c^\alpha_\beta(a)\NO{\varphi_\alpha(z)\varphi^*_\beta(z)},
\end{equation*}
where $c^\alpha_\beta(a)$ is defined by $[a,e_\beta]=\sum_{\alpha\in S}c_{\alpha}^\beta(a) e_\alpha$. 
Those fields play an important role in the structure of the $\W$-algebra $\WA$.

Let $\g^f$ be the centralizer of $f$ in $\g$, and set for $j \in \Z$, $\g^f_j := \g^f \cap \g_j$. By the theory of $\sl{2}$, we have 
\begin{align}\label{gradgf} 
\g^f = \bigoplus_{j \leq 0}\g^f_j.
\end{align}

\begin{Theorem}[{\cite[Theorem 4.1]{KW04}}]
\label{Wstructure}
For each $a\in\g^f_j$, $j\leq 0$, there exists a field $J^{\{a\}}(z)$ of $\WA$ of conformal weight $1- j$ 
with respect to $L$
such that $J^{\{a\}}(z)-J^{a}(z)$ is a linear combination of normally order products of the fields 
$J^b(z)$, where $b \in \g_{-s}$, $0 \leq s <j$ and their derivatives. 
	
Let $\{a_i\}_{i\in I}$ be a basis of $\g^f$ compatible with the graduation \eqref{gradgf}. 
Then $\WA$ is strongly generated by the fields $\{J^{\{a_i\}}\}_{i\in I}$.
\end{Theorem}

In practice, to construct $J^{\{a\}}(z)$ from $J^{a}(z)$, with $a \in \g^{f}$, 
we write a linear combination as in the theorem and try to find the coefficients so that the field $J^{\{a\}}(z)$ is $d_{(0)}$-closed. 

\section{Generators of $\W^k(\sp{4},f_{subreg})$ and OPEs}\label{sec:OPE}
From now on, $\g$ refers to the simple Lie algebra $\sp{4}$ that we may realize as the set of $4$-size square matrices $x$ such that $x^T J_4 + J_4 x =0$, where $J_4$ is the anti-diagonal matrix given by 
\[J_4 = \begin{pmatrix} 0&U_2\\-U_2&0\end{pmatrix},\qquad\text{where } U_2=\begin{pmatrix} 0&1\\1&0\end{pmatrix}.\]
We make the standard choice that $\h$ is the set of diagonal matrices of $\g$. 
Nilpotent orbits of $\g=\sp{4}$ are parametrized by the partitions of $4$ such that the number of parts of each odd  number is even (see, for instance, \cite[Theorem~5.1.3]{CM93}). 
Thus there are four nilpotent orbits in $\g=\sp{4}$ corresponding to the following partitions: $(4)$, $(2^2)$, $(2,1^2)$, $(1^4)$. 
They correspond, respectively, to the principal (or regular), the subregular, the minimal and the zero nilpotent orbits of $\g$, with respective dimensions $8$, $6$, $4$, $0$. 

Write $\Pi=\{\alpha_1,\alpha_2\}$ a set of simple roots for the root system $\Delta$ of $(\g,\h)$ such that $\alpha_1$ is a long root and $\alpha_2$ is short. 
Then $\Delta_+=\{\alpha_1,\alpha_2,\eta,\theta\}$, with $\eta:=\alpha_1+\alpha_2$ and $\theta:=\alpha_1+2\alpha_2$ the highest positive root. 
The centralizer of $e_{-\eta}$ is four-dimensional generated by $e_{-\eta},e_{-\alpha_1},e_{-\theta},h_2$, where $h_i := \alpha_i^\vee \in (\h^*)^* \cong \h$, for $i=1,2$. 
Hence  \[f:=e_{-\eta}=f_{subreg}\] belongs to the subregular nilpotent orbit of $\g$. 
Setting $e:=e_{\eta}$ and $h:=[e,f]=2h_1+h_2$ we get the $\sl{2}$-triple $(e,h,f)$ of $\g$.  
The nilpotent element $f$ is even and we have:
\[\g=\g_{-1}\oplus\g_0\oplus\g_1.\]
Moreover, $\g^f=\g_{-1}^f\oplus\g^f_0$, with $\g_{-1}^{f}=\C f \oplus \C e_{-\alpha_1}\oplus \C e_\theta$ and $\g^f_0= \C h_2.$ 
\begin{Remark}
\label{Rem:Levi_type}
The smallest Levi subalgebra of $\sp{4}$ containing $f=e_{-\eta}$ 
has semisimple type $A_1$ 
with basis $h_1,h_2, 
e_{\pm \eta}$ (it is the centralizer in $\sp{4}$ of $h_2$), 
and, hence, $f$ has Levi type since it is principal in it.
\end{Remark}

Using Theorem \ref{Wstructure} and the computer program {Mathematica} (there is a Mathematica package \cite{Thi91} which provides a computer program for the OPE calculations), we obtain that the vertex algebra $\WA$ is strongly generated by the fields $J(z)$, $G^\pm(z)$ and $L(z)$, provided that $k \neq -3$, where: 
\begin{align*}
J(z)&= J^{\{h_2\}}(z)= \frac{1}{2}J^{h_2}(z),\\
G^+(z)&= J^{\{e_{-\alpha_1}\}}(z)= J^{e_{-\alpha_1}}(z)+\frac{1}{2}(\NO{J^{h_1}(z)J^{e_{\alpha_2}}}+(k+2)\partial J^{e_{\alpha_2}}(z)),\\
G^-(z)&= J^{\{e_{-\theta}\}}(z)= J^{e_{-\theta}}(z)+\frac{1}{2}(\NO{J^{h_1}(z)J^{e_{-\alpha_2}}(z)}+\NO{J^{h_2}(z)J^{e_{-\alpha_2}}(z)}+(k+2)\partial J^{e_{-\alpha_2}}(z)),\\
L(z)&=J^{\{f\}}(z)=\frac{1}{(3+k)}(-J^{f}(z)+\frac{1}{2}(\NO{J^{e_{\alpha_2}}(z)J^{e_{-\alpha_2}}(z)}+\NO{J^{h_1}(z)^2}+\NO{J^{h_1}(z)J^{h_2}(z)}\\&\qquad+(5+2k)\partial J^{h_1}(z))+\NO{J(z)^2}+(1+k)\partial J(z)).
\end{align*}
These fields satisfy the following OPE's :
\begin{align*}
J(z)J(w)\sim&\OPE{2}{(2+k)}{},\\
J(z)G^\pm(w)\sim&\pm\OPE{}{1}{}G^\pm(w),\\
L(z)L(w)\sim&\OPE{4}{\charge}{2}+\OPE{2}{2}{}L(w)+\OPE{}{1}{}\partial L(w),\\
L(z)G^\pm(w)\sim&\OPE{2}{2}{}G^\pm(w)+\OPE{}{1}{}\partial G^\pm(w),\\
L(z)J(w)\sim&\OPE{2}{1}{}J(w)+\OPE{}{1}{}\partial J(w),\\
G^\pm(z)G^\pm(w)\sim&0,\\
G^+(z)G^-(w)\sim&-\OPE{4}{3(1+k)(2+k)^2}{}-\OPE{3}{3(1+k)(2+k)}{}J(w)\\&+\OPE{2}{1}{}\left((2+k)(3+k)L(w)-(3+2k)\NO{J(w)^2}-\frac{3(1+k)(2+k)}{2}\partial J(w)\right)\\&+\OPE{}{1}{}\left((3+k)\NO{L(w)J(w)}+\frac{(3+k)(2+k)}{2}\partial L(w)-\NO{J(w)^3}\right.\\&\left.-(3+2k)\NO{J(w)\partial J(w)}-\frac{(5+4k+k^2)}{2}\partial^2 J(w)\right),
\end{align*}
where 
\[\charge :=-\frac{2(9+16k+6k^2)}{3+k}.\]

The field $L(z)=\sum_{n\in\Z}L_nz^{-n-2}$ is a conformal vector of $\WA$ with central charge $\charge$. 
It gives $J(z)$, $G^+(z)$ and $G^-(z)$ the conformal weights $1$, $2$ and $2$, respectively. 
Following \eqref{eq:conformal_fileds} we write 
\begin{align*}
J(z)&=\sum_{n\in\Z}J_nz^{-n-1},& G^\pm(z)&=\sum_{n\in\Z}G^\pm_nz^{-n-2},& L(z)&=\sum_{n\in\Z}L_nz^{-n-2}.
\end{align*}
The monomials 
\begin{equation}\label{PBW}
	J_{n_1}\ldots J_{n_j}L_{m_1}\ldots L_{m_t}G^-_{p_1}\ldots G^-_{p_g}G^+_{q_1}\ldots G^+_{q_h}\vac,
\end{equation}
with $n_1\leq\ldots \leq n_j\leq-1$, $m_1\leq\ldots m_t\leq -2$, $p_1\leq\ldots \leq p_g\leq-2$ and $q_1\leq\ldots \leq q_h\leq-2$, 
form a PBW basis of the vertex algebra $\WA$. 

From the OPEs, we deduce the commutation relations for all $m,n\in\Z$:
\begin{align*}
[J_m,J_n]=&(2+k)m\indic{m+n,0},\\
[J_m,G^\pm_n]=&\pm G^\pm_{m+n},\\
[L_m,L_n]=&\frac{\charge}{12} (m^3-m)\indic{n+m,0}+(m-n)L_{m+n},\\
[L_m,G^\pm_n]=&(m-n)G^\pm_{m+n},\\
[L_m,J_n]=&-nJ_{m+n},\\
[G^+_m,G^-_n]=&-\frac{(1+k)(2+k)^2}{2}(m^3-m)\indic{m+n,0}+\frac{(2+k)(3+k)}{2}(m-n)L_{m+n}\\
&\,+\left(\frac{3(1+k)(2+k)}{2}(m+1)(n+1)-\frac{(5+4k+k^2)}{2}(m+n+1)(m+n+2)\right)J_{m+n}\\
&\,-(3+2k)(m+1)(J^2)_{m+n}+(3+k)(LJ)_{m+n}-(J^3)_{m+n}-(3+2k)(J\partial J)_{m+n},
\end{align*}
where \begin{align*}\sum_{n\in\Z}(J^2)_nz^{-n-2}&\overset{\text{def}}{=}\NO{J(z)^2}\,, &\sum_{n\in\Z}(LJ)_nz^{-n-3}&\overset{\text{def}}{=}\NO{L(z)J(z)}\,,& \sum_{n\in\Z}(J\partial J)_nz^{-n-3}&\overset{\text{def}}{=}\NO{J(z)\partial J(z)}.\end{align*}

\section{The twist $\psi$-action over simple $\W_k(\sp{4},f_{subreg})$-modules}\label{sec:twist}
We continue to assume that $\g=\sp{4}$ and $f=f_{subreg}=e_{-\eta}$ and we keep the notation of the previous section. 
We assume for the moment that $k$ is any non-critical level for $\g$, that is, $k \not=-3$. 

Let $M$ be an irreducible positive energy representation of $\W^k(\g,f)$, with $M_{top}=M_\chi$, $\chi \in \C$ (see Sect.~\ref{sec:vertex_algebras}). 
Using the commutation relations of Sect.~\ref{sec:OPE}, we notice that the submodule $N$ of $M$ 
generated by all $J_0$-eigenvectors of $M$ is stable by $L_n, G_n^{\pm}, J_n$, $n \in \Z$. 
Hence, $N=M$ because $M$ is simple. 
\begin{Lemma}\label{Lem:existence} 
Let $M$ be an irreducible positive energy representation of $\WA$, with $\top{M}=M_\chi$, $\chi\in\C$. 
Suppose that $\top{M}$ is finite dimensional.
Then, there is a vector $v \in M$ such that $L_0 v = \chi v$, $J_0 v= \xi v$ for some $\xi \in \C$, and such that the below relations hold:
\begin{align*}
J_n v=0 \quad \text{ for } n>0,\\
L_n v=0\quad\text{ for } n>0,\\
G^+_n v=0\quad\text{ for } n >0,\\
G^-_n v=0\quad\text{ for } n\geq 0.
\end{align*}
Moreover, 
$M=\bigoplus\limits_{a\in \xi+\Z\atop d \in \chi+\Z_{\geq 0}} M_{a,d}$, 
where $M_{a,d}=\{m \in M\colon J_0 m = a m, \, L_0 m = d m\}$, $\dim M_{\xi,\chi}=1$ and  $M_{top} = M_\chi$ is spanned by the vectors $(G^+_0)^i v $ for $i\geq0$. 
\end{Lemma}

\begin{proof}
Since  $J_0$ and $L_0$ commute, the action of $J_0$ preserve each $M_{\chi+n}$, $n\in\Z_{\geq0}$. 
Moreover $J_0$ is semisimple over $M$ and each $M_n$, then $M$ can be written 
as 
$${M=\bigoplus\limits_{(a,d) \in \C^2 \atop d \in \chi+\Z_{\geq 0}} M_{a,d}}.$$
As $\top{M}$ is finite dimensional, there is an vector $v\in\top{M}$ such that $J_0v=\xi v$, $\xi\in\C$ and $\xi-n$ is not an eigenvalue of $J_0$ in $\top{M}$ for all $n\in\Z_{> 0}$.
The relations of the lemma result from the following equations. Let $n\in\Z$,
\begin{align*}
	&J_0J_nv=\xi J_nv,&&L_0J_nv=(\chi-n)J_nv,\\
	&J_0L_nv=\xi L_nv,&&L_0L_nv=(\chi-n)L_nv,\\
	&J_0G^\pm_nv=(\xi\pm1)G^\pm_nv,&&L_0G^\pm_nv=(\chi-n)G^\pm_nv.
\end{align*}
It ensues from those relations that all the eigenvalues of $J_0$ are in $\xi+\Z$. Hence,
$$M=\bigoplus\limits_{a\in \xi+\Z\atop d \in \chi+\Z_{\geq 0}} M_{a,d}.$$
We explain the relation $G^-_n v=0$, $n\geq 0$, the others are obtained similarly.
Since $M$ is a positive energy representation of $\WA$, for $n>0$, $\chi-n$ is not an eigenvalue of $L_0$ and $G^-_nv=0$. 
Besides, $G^-_0v\in\top{M}$ and the choice of $v$ implies that $\xi-1$ is not a eigenvalue of $J_0$ in $\top{M}$. 
Hence $G^-_0v=0$.
The vectors $(G^+_0)^i v $, $i\geq0$, are the only ones attached 
to the eigenvalue $\chi$ for $L_0$ and $J_0(G^+_0)^iv=(\xi+i)G^+_0v$ for $i\geq0$. As a consequence they span $\top{M}$ and $M_{\xi,\chi}=\C v$.
\end{proof}

For $(\xi,\chi) \in \C^2$, let $L(\xi,\chi)$ be the irreducible representation of $\W^k(\g,f)$ generated by a vector $v=|\xi,\chi\rangle$ 
satisfying the relations of Lemma \ref{Lem:existence}. 
According to Lemma \ref{Lem:existence}, $|\xi,\chi\rangle$ is uniquely defined up to nonzero scalar, 
so the notation is legitimate. 
Zhu's correspondence ensures that such $L(\xi,\chi)$ does exist and is unique up to isomorphism of $\W^k(\g,f)$-modules 
(see, for example, \cite{AM20}). 

Since $G^-_0G^+_0\vectop $ is in $\Lxichi_{\chi,\xi}$, it is proportional to $\vectop $, and we easily obtain that 
\begin{equation*}
	G^-_0G^+_0\vectop =g(\xi,\chi)\vectop ,
\end{equation*}
where 
\[g(\xi,\chi)=-\frac{1}{2}\xi\left(2+3k+k^2-2\xi^2+6\chi+2k\chi\right).\]
Hence for $i\geq1$,
\begin{equation}\label{G-G+i}
	G^-_0(G^+_0)^i\vectop =i\,h_i(\xi,\chi)(G^+_0)^{i-1}\vectop ,
\end{equation}
where 
\begin{align*} 
h_i(\xi,\chi) & =\frac{1}{i}\sum_{m=0}^{i-1}g(\xi+m,\chi) & \\
& = \frac{(2\xi+i-1)}{4}(-2-i+i^2-3k-k^2-2\xi+2i\xi+2\xi^2-6\chi-2k\chi). &
\end{align*}

\begin{Proposition}
Suppose that $\top{\Lxichi}$ is $n$-dimensional. Then $h_n(\xi,\chi)=0$.
\end{Proposition}

\begin{proof}
If $\dim\top{\Lxichi}=n$ then $(G^+_0)^n\vectop =0$ and $(G^+_0)^{n-1}\vectop \neq0$. 
It results from \eqref{G-G+i} that $h_n(\xi,\chi)=0$.
\end{proof}

Following the ideas of \cite{Ara13}, we introduce the twist-action $\psi$ (\cite{Li97}).
Let define
\begin{equation*}
	\Delta(-J,z) :=z^{-J_0}\exp\left(\sum_{m=1}^{\infty}(-1)^{m+1}\frac{-J_m}{mz^m}\right).
\end{equation*}
For $a\in \WA$,
\begin{equation*}
	\Delta(-J,z)a=z^{-J_0}\left(\sum_{n=0}^{\infty}\frac{X^n}{n!}a\right),
\end{equation*}
where $X=\sum\limits_{m=1}^{\infty}(-1)^{m+1}\frac{-J_m}{kz^m}$ and $z^{-J_0}$ is defined by $z^{-J_0}a=z^{-c}a$ if $J_0a=ca$.
Set
\begin{equation*}
\sum_{n\in\Z}\psi(a_{(n)})z^{-n-1}:=
Y(\Delta(-J,z)a,z)=\sum_{n=0}^{\infty}Y(z^{-J_0}\frac{X^n}{n!}a,z).
\end{equation*}
For any $\WA$-module $M$, the space $\psi(M)$ denotes the $\WA$-module obtained by twisting the action of $\WA$ as $a_{(n)}\mapsto\psi(a_{(n)})$. 
The following relations are obtained by applying the $\psi$-action to the generators of $\WA$:
\begin{align*}
\psi(J_n)&=J_n-(2+k)\indic{n,0},\\
\psi(L_n)&=L_n-J_n+\frac{(2+k)}{2}\indic{n,0},\\
\psi(G^+_n)&=G^+_{n-1},\\
\psi(G^-_n)&=G^-_{n+1}.
\end{align*}

\begin{Proposition}
Assume that $\dim\top{\Lxichi}=i$ and $\dim\top{\psi(\Lxichi)}=j$. Then
\begin{equation*}
\psi^2(\Lxichi)\simeq L(\xi+i+j-6-2k\,,\,\chi-2\xi-2i-j+7+2k).
\end{equation*}
\end{Proposition}

\begin{proof}
For all $m\geq0$ we have
\begin{align*}
\psi(J_0)(G^+_0)^m\vectop &=(\xi+m-(2+k))(G^+_0)^m\vectop ,\\
\psi(L_0)(G^+_0)^m\vectop &=(\chi-(\xi+m)+\frac{(2+k)}{2})(G^+_0)^m\vectop .
\end{align*}
Since the smallest eigenvalue associated with the $\psi(L_0)$-action is attached to the vector $(G^+_0)^{i-1}\vectop $, we get 
\begin{equation*}
\psi(\Lxichi)\simeq L(\xi+(i-1)-(2+k)\,,\,\chi-\xi-(i-1)+\frac{(2+k)}{2})
\end{equation*}
and
\begin{align*}
\psi^2(\Lxichi)&\simeq\psi(L(\xi+(i-1)-(2+k)\,,\,\chi-\xi-(i-1)+\frac{(2+k)}{2}))\\
&\simeq L(\xi+(i-1)+(j-1)-2(2+k)\,,\,\chi-2\xi-2(i-1)-(j-1)+2(2+k)).
\end{align*}
\end{proof}

\begin{Remark}
For all $m,n\in\Z_{\geq 0}$,
\begin{align*}
\psi^2(J_0)(G^+_{-1})^m(G^+_0)^n\vectop &=(\xi+n+m-2(2+k))(G^+_{-1})^m(G^+_0)^n\vectop ,\\
\psi^2(L_0)(G^+_{-1})^m(G^+_0)^n\vectop &=(\chi-2\xi-2n-m+2(2+k))(G^+_{-1})^m(G^+_0)^n\vectop .
\end{align*}
\end{Remark}

\begin{Proposition}\label{2.4}
	Suppose that $\dim\top{\Lxichi}=i$, $\dim\top{\psi(\Lxichi)}=j$ and ${\dim\top{\psi^2(\Lxichi)}=l}$.
		\begin{enumerate}[font=\upshape,label=(\alph*)]
		\item \label{2.4P} 
		If $k=-3+p/3$ with $\gcd{p}{3}=1$, $p\geq3$ 
		then $(\xi,\chi,l)=(\xi^{(s)}_{i,j},\chi^{(s)}_{i,j},l^{(s)}_{i,j})$ 
		with $s \in \{1,2,3\}$, where
		\end{enumerate}
		\begin{gather*}
			\xi^{(1)}_{i,j}=\frac{1-i}{2},\quad \chi^{(1)}_{i,j}=\frac{13-6i+i^2-12j+2ij+2j^2+6k-2ik-4jk}{4(3+k)},\\ 
			l^{(1)}_{i,j}=9-i-j+3k,\\
			\xi^{(2)}_{i,j}=\frac{7-2i-j+2k}{2},\quad \chi^{(2)}_{i,j}=\frac{31-12i+2i^2-12j+2ij+j^2+18k -4ik-4jk+2k^2}{4(3+k)},\\ 
			l^{(2)}_{i,j}=i,\\
			\xi^{(3)}_{i,j}=\frac{4-i-j+k}{2},\quad \chi^{(3)}_{i,j}=\frac{4+i^2-6j+j^2-2jk-k^2}{4(3+k)},\\ 
			l^{(3)}_{i,j}=9-i-j+3k.
		\end{gather*}
	 	\begin{enumerate}[font=\upshape,label=(\alph*),resume]
		\item \label{2.4CP} If $k=-3+p/4$ 
		with $\gcd{p}{2}=1$, $p\geq4$ then $(\xi,\chi,l)=(\xi^{(s')}_{i,j},\chi^{(s')}_{i,j},l^{(s')}_{i,j})$ 
		with $s \in \{1,2\}$, where
		\end{enumerate}
		\begin{gather*}
			\xi^{(1')}_{i,j}=\frac{1-i}{2},\quad 
			\chi^{(1')}_{i,j}=\frac{13-6i+i^2-12j+2ij+2j^2+6k-2ik-4jk}{4(3+k)},\\ 
			l^{(1')}_{i,j}=12-i-2j+4k,\\
			\xi^{(2')}_{i,j}=\frac{7-2i-j+2k}{2},\quad 
			\chi^{(2')}_{i,j}=\frac{31-12i+2i^2-12j+2ij+j^2+18k-4ik-4jk+2k^2}{4(3+k)},\\ 
			l^{(2')}_{i,j}=i.
		\end{gather*}
\end{Proposition}

\begin{proof}
By solving the system of equations
\begin{equation*}
\left\{\begin{aligned}
&h_i(\xi,\chi)=0,\\
&h_j(\xi+(i-1)-(2+k),\chi-\xi-(i-1)+\frac{(2+k)}{2})=0,\\
&h_l(\xi+(i-1)+(j-1)-2(2+k),\chi-2\xi-2(i-1)-(j-1)+2(2+k))=0.
\end{aligned}\right.
\end{equation*}
we find nine triples $(\xi,\chi,l)$ in term of $i$, $j$ and~$k$. Since $l$ is the dimension of  
${\top{\psi^2(\Lxichi)}}$, it must be a positive integer. If $k=-3+p/3$, with $\gcd{p}{3}=1$, $p\geq3$, the three triples described in the first part of the proposition are the only ones among the solutions of the system corresponding to this restrictive condition. Similarly, if $k=-3+p/4$, $\gcd{p}{2}=1$, $p\geq4$, we find that only two triples verify the condition.
\end{proof} 

\begin{Proposition}\label{2.5}
\begin{enumerate}[font=\upshape,label=(\alph*)]
\item Let $k=-3+p/3$ with $\gcd{p}{3}=1$, $p\geq3$ then $(G^+_{-2})^{p-2}\vac$ belongs to the maximal ideal of $\WA$.
\item Let $k=-3+p/4$ with $\gcd{p}{2}=1$, $p\geq4$ then $(G^+_{-2})^{p-3}\vac$ belongs to the maximal ideal of $\WA$.
\end{enumerate}
\end{Proposition}

\begin{proof}
\ref{2.4P} For $i=j=1$, we have $l^{(1)}_{1,1}=p-2$. 
Since $\xi^{(1)}_{1,1}=\chi^{(1)}_{1,1}=0$ and $L_0\vac=J_0\vac=0$ the correspondence $\vac \mapsto {\vacij{1,1}{1}}$  yields the isomorphism 
\[\SWA\simeq \Lxichiij{1,1}{1}.\]
Moreover $\top{\psi^2(\W_k(\g,f))}$ is at most $p-2$ dimensional because 
\[h_{p-2}(-2(2+k),2(2+k))=0.\]
Hence $(G^+_{-2})^{p-2}\vac=\psi^2((G^+_0)^{p-2})\vac=0$.
	
\ref{2.4CP} The argument is the same as in the previous case with $l^{(1')}_{1,1}=p-3$ and ${\xi^{(1')}_{1,1}=\chi^{(1')}_{1,1}=0}$. 
\end{proof}

Assume that either $k=-3+p/3$, with $\gcd{p}{3}=1$, $p\geq3$, or $k=-3+p/4$, with $\gcd{p}{2}=1$, $p\geq4$. 
Then $k$ is admissible for $\g$ and by our choice of $f=f_{subreg}$, we have $f \in \mathbb{O}_q$, with $q$ the denominator of $k+3$.  
So by Proposition \ref{Pro:associated_variety_W-algebra}, the simple quotient $\SWA$ of $\WA$ is lisse.  
As a consequence, any simple $\SWA$-module is positively graded with finite dimensional 
top degree component and so is of the form $L(\xi,\chi)$ for some $(\xi,\chi)\in\C^2$. 

We are now in a position to state the main result of this section. 

\begin{Proposition}\label{2.7}
Let $M$ be a simple $\SWA$-module.
\begin{enumerate}[font=\upshape,label=(\alph*)]
\item \label{2.7P} If $k=-3+p/3$ with $\gcd{p}{3}=1$, $p\geq3$,  
then the $\SWA$-module $M$ is isomorphic to $\Lxichiij{i,j}{s}$ for $\xi^{(s)}_{i,j}$ and $\chi^{(s)}_{i,j}$ as in 
Proposition \ref{2.4}\ref{2.4P} with $1\leq i\leq p-2$, $1\leq j\leq p-i-1$ and $s\in\{1,2,3\}$.
\item \label{2.7CP} If $k=-3+p/4$ with $\gcd{p}{2}=1$, $p\geq4$,  
then the $\SWA$-module $M$ is isomorphic to $\Lxichiij{i,j}{s'}$ for $\xi^{(s')}_{i,j}$ and $\chi^{(s')}_{i,j}$ as in 
Proposition \ref{2.4}\ref{2.4CP} with $1\leq i\leq p-3$ and $1\leq j\leq (p-i-1)/2$ if $s=1$ or $1\leq i\leq p-3$ and $1\leq j\leq p-2i-1$ if $s=2$.
\end{enumerate}	
\end{Proposition}

\begin{proof}
\ref{2.7P} Since $M$ is a simple $\SWA$-module, there exist $\xi,\chi\in\C$ such that $M \cong \Lxichi$.  
By Proposition \ref{2.5}, $G^+(z)^{p-2}=0$. 
Hence 
\[\NO{G^+(z)^{p-2}}\overset{\text{def}}{=}\sum_{n\in\Z}((G^+)^{p-2})_nz^{-n-2}=0.\]
In particular $(G^+_0)^{p-2}\vectop =((G^+)^{p-2})_0\vectop =0$. 
As a consequence $\top{\Lxichi}$ is at most $(p-2)$-dimensional.
By Proposition \ref{2.4}\ref{2.4P}, since $\psi^2(\Lxichi)$ is a simple $\SWA$-module, there exist $1\leq i,j\leq p-2$ and $s\in\{1,2,3\}$ such that $\xi=\xi^{(s)}_{i,j}$ and $\chi=\chi^{(s)}_{i,j}$. 
In the same way, $\psi^4(\Lxichi)$ is a simple $\SWA$-module and there are $1\leq l,m\leq p-2$ and $r\in\{1,2,3\}$ such that $\psi^2(\xi):=\xi+(i-1)+(j-1)-2(2+k)=\xi^{(r)}_{l,m}$ and $\psi^2(\chi):=\chi-2\xi-2(i-1)-(j-1)+2(2+k)=\chi^{(r)}_{l,m}$. 
The $\psi^2$-action permutes the three forms of the eigenvalues $\xi$ and $\chi$:
\begin{align*}
\psi^2(\Lxichiij{i,j}{1})&\simeq\Lxichiij{p-i-j,i}{2},\\
\psi^2(\Lxichiij{i,j}{2})&\simeq\Lxichiij{i,p-i-j}{3},\\
\psi^2(\Lxichiij{i,j}{3})&\simeq\Lxichiij{p-i-j,j}{1}.
\end{align*}
The condition $j\leq p-i-1$ comes from $1\leq l,m\leq p-2$.
		
\ref{2.7CP} The argument is quite similar. 
By Proposition \ref{2.4}\ref{2.4CP}, $G^+(z)^{p-3}=0$ and $\top{\Lxichi}$ is at most $(p-3)$-dimensional. 
Moreover since $\psi^2(\Lxichi)$ and $\psi^4(\Lxichi)$ are simple $\SWA$-modules, 
$\xi=\xi^{(s')}_{i,j}$, $\chi=\chi^{(s')}_{i,j}$, $\psi^2(\xi)=\xi^{(r')}_{l,m}$ and $\psi^2(\chi)=\chi^{(r')}_{l,m}$ 
with $1\leq i,j,l,m\leq p-3$, $r,s\in\{1,2\}$. 
On the contrary of the first case, the $\psi^2$-action preserves the form of the eigenvalues $\xi$ and $\chi$:
\begin{align*}
\psi^2(\Lxichiij{i,j}{1'})&\simeq\Lxichiij{p-i-2j,j}{1'},\\
\psi^2(\Lxichiij{i,j}{2'})&\simeq\Lxichiij{i,p-2i-j}{2'}.
\end{align*}
If $s=1$ then the condition $p-i-2j\geq1$ implies $j\leq\frac{p-i-1}{2}$, 
and if $s=2$, with the same argument we get $j\leq p-2i-1$.
\end{proof}

\begin{Remark}
The simple $\SWA$-modules $\Lxichiij{i,j}{s}$ of Proposition~\ref{2.7} are all mutually non-isomorphic since their highest weights are distinct.
\end{Remark}

\begin{Remark}
For $k=-3+p/3$, with $\gcd{p}{3}=1$, $p\geq3$, or $k=-3+p/4$, with $\gcd{p}{2}=1$, $p\geq4$, the application $\psi$ is a bijection of the set of the simple $\SWA$-modules $\Lxichiij{i,j}{s}$ described in Proposition~\ref{2.7} over itself of inverse $\psi^5$ if $k$ is principal admissible, and $\psi^3$ otherwise. 
We describe below the $\Lxichiij{i,j}{s}$ orbits under the $\psi$-action:
	\begin{itemize}
		\item if $k=-3+p/3$ with $(p,3)=1$, $p\geq3$ then
	\end{itemize}
		\begin{gather*}
		L(\xi^{(1)}_{i,j},\chi^{(1)}_{i,j})\overset{\psi}{\to} L(\xi^{(3)}_{j,p-i-j},\chi^{(3)}_{j,p-i-j})\overset{\psi}{\to} L(\xi^{(2)}_{p-i-j,i},\chi^{(2)}_{p-i-j,i})\overset{\psi}{\to} L(\xi^{(1)}_{i,p-i-j},\chi^{(1)}_{i,p-i-j})\\ \overset{\psi}{\to} L(\xi^{(3)}_{p-i-j,j},\chi^{(3)}_{p-i-j,j})\overset{\psi}{\to} L(\xi^{(2)}_{j,i},\chi^{(2)}_{j,i}))\overset{\psi}{\to} L(\xi^{(1)}_{i,j},\chi^{(1)}_{i,j}).
		\end{gather*}
	\begin{itemize}
		\item if $k=-3+p/4$ with $(p,2)=1$, $p\geq4$ then
	\end{itemize}
		\begin{gather*}
		L(\xi^{(1')}_{i,j},\chi^{(1')}_{i,j})\overset{\psi}{\to} L(\xi^{(2')}_{j,p-i-2j},\chi^{(2')}_{j,p-i-2j})\overset{\psi}{\to} L(\xi^{(1')}_{p-i-2j,j},\chi^{(1')}_{p-i-2j,j})\\ 
		\overset{\psi}{\to} L(\xi^{(2')}_{j,i},\chi^{(2')}_{j,i})\overset{\psi}{\to} L(\xi^{(1')}_{i,j},\chi^{(1')}_{i,j}).
		\end{gather*}
\end{Remark}	

\section{Proof of the Main Theorem}\label{sec:proof}
The section is devoted to the proof of the Main Theorem. 
As explained in the introduction, the last step uses results and technics from \cite{AvE19}. 

For the moment, $\g$ is any simple Lie algebra and $f$ 
is a nilpotent element of $\g$.

Let 
\[P_{0,+}:=\{\lambda\in\h^*\colon \langle \lambda, \alpha^\vee \rangle \in\Z_{\geq 0} \text{ for all } \alpha\in\Delta_{0,+}\},\]
where $\Delta_{0,+}:=\Delta_0\cap \Delta_+$  with $\Delta_0$ the root system of $(\g_0,\h)$. 
Note that $P_{0,+}$ contains the set 
\[P_+ =\{\lambda\in\h^* \colon \langle \lambda, \alpha^\vee\rangle 
\in\Z_{\geq 0} \text{ for all }\alpha\in\Delta_+\} .\]

Let $U(\g)$ be the enveloping algebra of $\g$ and $H_f^0(?)$ the functor from the category of Harish-Chandra $U(\g)$-bimodules to the category of bimodules over the {\em finite $\W$-algebra $U(\g, f)$} (see \cite{Pre02}), which is the Zhu algebra of $\WA$ \cite{DK06}. 
By abuse of notation we use the same notation as for the functor in Sect.~\ref{sec:BRST}. 
We refer to \cite{AvE19} for more about this topic.

Write $J_\lambda\subset U(\g)$ for the annihilating ideal of the irreducible $\g$-module $L(\lambda)$ with highest weight $\lambda\in\h^*$. 
The quotient $H^0_f(U(\g)/J_\lambda)$ is a quotient algebra of the finite $W$-algebra $U(\g,f)=H^0_f(U(\g))$. 
For $\lambda\in P_{0,+}$, such that $\lambda + k\Lambda_0$ is admissible and $\dim L(\lam)$ is maximal, 
$H^0_f(U(\g)/J_\lambda)$ has a unique simple module denoted by $E_{J_\lambda}$ \cite[Theorem 7.7]{AvE19}. 
Let ${\mathbf{L}}(E_{J_\lambda})$ be the irreducible {\em Ramond twisted $\W_k(\g,f)$-module} attached to $E_{J_{\lambda}}$ (see \cite{KW08,Ara11}). The module ${\mathbf{L}}(E_{J_\lambda})$ is the unique simple quotient of the Verma module  \cite{AvE19}, 
$${\mathbf{M}(E_{J_\lambda}):=U(\WA)\*_{U(\WA)_{\geq0}}E_{J_\lambda}}.$$

\begin{Theorem}[{\cite[Theorem 7.8]{AvE19}}]
\label{th7.8}
	Let $k=-h^\vee+p/q$ be an admissible number for $\g$ and pick $f\in\O_q$. 
	Let $\lambda\in P_{0,+}$ be such that $\hat{\lambda}=\lambda + k\Lambda_0 \in Pr^k$. 
	Then 
	\[H^0_f(\widehat{L}_k({\lam}))\simeq {\mathbf L}(E_{J_{\lambda-\frac{p}{q}x_0}}),\]
	where $\widehat{L}_k({\lam})$ is the simple $\affg$-module with highest weight $\hat{\lambda}$.
	In particular, 
	\[\SWA\simeq H^0_f(\SVA)\simeq {\mathbf L}(E_{J_{-\frac{p}{q}x_0}}).\]
\end{Theorem}

Assume now, as in the previous sections, that $\g=\sp{4}$ and $f=f_{subreg}=e_{-\eta}$. 
Here $\Delta_{0,+} = \{\alpha_2\}$, whence 
\[P_{0,+}=\{\lambda\in\h^* \colon \langle \lambda, \alpha_2^\vee \rangle 
\in\Z_{\geq 0}\}=\C\varpi_1+\Z_{\geq 0}\varpi_2,\]
and
\[P_+=\Z\varpi_1+\Z\varpi_2,\]
where $\varpi_1,\varpi_2\in\h^*$ are the dual elements of $\alpha_1^\vee$ and $\alpha_2^\vee$ respectively : for $i,j\in\{1,2\}$, $\varpi_i(\alpha_j^\vee)=\indic{i,j}$.
According Proposition \ref{2.7} if $M$ is a simple $\SWA$-module, where $k=-3+p/3$, with $\gcd{p}{3}=1$, $p\geq3$, or $k=-3+p/4$, with $\gcd{p}{2}=1$, $p\geq4$,  
then it is isomorphic to $\Lxichiij{i,j}{s}$ with $\xi^{(s)}_{i,j}$ and $\chi^{(s)}_{i,j}$ 
as described in Proposition~\ref{2.4}. 

\begin{Proposition}\label{correspondence}
	Let  $k$, $\xi^{(s)}_{i,j}$ and $\chi^{(s)}_{i,j}$ be as in Proposition \ref{2.7} then $\Lxichiij{i,j}{s}$ is a simple $\SWA$-module.
\end{Proposition}

\begin{proof}
The computation depends on whether $k$ is principal admissible of coprincipal admissible, but the argument is very similar in both cases. 
We only detail the case $k=-3+p/3$, with $\gcd{p}{3}=1$, $p \geq 3$. 
For $1\leq i\leq p-2$ and $1\leq j\leq p-i-1$, set $\lambda_{i,j}=(j-1)\varpi_1+(i-1)\varpi_2$. 
We check that $\lambda_{i,j}\in P_+\cap Pr^k$. 
By Theorem \ref{th7.8}, 
\[H^0_f(\widehat{L}_k({\lambda}_{i,j}))\simeq \mathbf{L}(E_{J_{\lambda_{i,j}-\frac{p}{3}x_0}})\simeq {H^0_{f,-}(\widehat{L}_k({\lambda}_{i,j}-\frac{p}{3}x_0))},\]
where $x_0=\frac{h}{2}=\varpi_1$ and $H^\bullet_{f,-}(?)$ is the ``$-$" variant of the quantized Drinfeld-Sokolov reduction functor $H^\bullet_{f}$ defined in Sect.~\ref{sec:BRST} (see \cite{Ara11,AvE19} for a construction).
The lowest $L_0$-eigenvalue $h_{\lambda_{i,j}-\frac{p}{3}x_0}$ is the conformal dimension of $\mathbf{L}(E_{J_{\lambda_{i,j}-\frac{p}{3}x_0}})$ given by \cite[(7.4)]{AvE19}:
\begin{align*}
h_{\lambda_{i,j}-\frac{p}{3}x_0}&=\frac{(\lam_{i,j}-\frac{p}{3}x_0|\lam_{i,j}-\frac{p}{3}x_0+2\rho)}{2(k+h^\vee)}-\frac{k+h^\vee}{2}\norm{x_0}^2+(x_0|\rho)\\
&=\frac{-15+3i^2+6ij+6j^2+6p-2ip-4jp}{4p}=\chi^{(1)}_{i,j}.
\end{align*}
Besides, using \cite[(70)]{Ara11} the lowest $J_0$-eigenvalue of $H^0_{f,-}(\widehat{L}_k({\lambda}_{i,j}))$ is
\[(\lambda_{i,j}-\frac{p}{3}x_0|-\frac{\alpha_2^\vee}{2})=\frac{1-i}{2}=\xi^{(1)}_{i,j},\]
since $J(z)=J^{\alpha_2}(z)$ and $\widehat{w}_0\widehat{t}_{-x_0}J_{\alpha_2}^R=-J_{\alpha_2}$.
In conclusion
\begin{equation}\label{simple_mod}
\mathbf{L}(E_{J_{\lambda_{i,j}-\frac{p}{3}x_0}})\simeq L(\xi^{(1)}_{i,j},\chi^{(1)}_{i,j}).
\end{equation}
In this way $\Lxichiij{i,j}{1}$ is a simple $\SWA$-module for all $1\leq i\leq p-2$ and $1\leq j\leq p-i-1$. If $s=2$ or $3$, using the $\psi$-action on the module $\Lxichiij{i,j}{s}$ it always comes down to a module $\Lxichiij{i',j'}{1}$. As a consequence $\Lxichiij{i,j}{s}$ is a simple module of $\SWA$, too. 	
\end{proof}

\begin{Lemma}
Suppose that there is a nontrivial extension of $\SWA$-modules, 
\[0\longrightarrow \Lxichi\overset{\iota}{\longrightarrow} M\overset{\pi}{\longrightarrow} \Lxichiprime\longrightarrow 0.\] Then $L_0$ acts locally finitely on $M$.
\end{Lemma}

\begin{proof}
	Suppose there is a nontrivial extension
		\begin{equation}\label{sequence}
			0\longrightarrow \Lxichi  \overset{\iota}{\longrightarrow} M \overset{\pi}{\longrightarrow} \Lxichiprime \longrightarrow 0.
		\end{equation}
		
	Since $\W_k(\g,f)$ is lisse, $L:=\Lxichi$ and $L':=\Lxichiprime$ 
	are $L_0$-diagonalizable and the $L_0$-eigenspaces are finite dimensional.
	
	Let $m\in M$. Since $\pi(m)\in L'$ there exist $w_1,\ldots, w_s\in L'$ and $\mu_1,\ldots,\mu_s\in\C$ such that $L_0w_j=\mu_jw_j$ for all $1\leq j\leq s$ and $\prod_{j=1}^{s}(L_0-\mu_j\id)\pi(m)=0$. Then
	\[\pi(\prod_{j=1}^{s}(L_0-\mu_j\id)m)=0.\]
	As a consequence $\prod_{j=1}^{s}(L_0-\mu_j\id)m\in\Ima\iota$. Let $m_1\in L$ such that $\iota(m_1)=\prod_{j=1}^{s}(L_0-\mu_j\id)m$. As before, there are $v_1,\ldots,v_r\in L$ and $\nu_1,\ldots,\nu_r\in\C$ 
	such that $\prod_{i=1}^{r}(L_0-\nu_i\id)m_1=0$. 
	Then
	\[\iota(\prod_{i=1}^{r}(L_0-\nu_i\id)m_1)=0=\prod_{i=1}^{r}(L_0-\nu_i\id)\iota(m_1)=\prod_{i=1}^{r}(L_0-\nu_i\id)\prod_{j=1}^{s}(L_0-\mu_j\id)m.\]
	Hence $m$ belongs to some $L_0$-stable finite dimensional vector subspace of
	\[\bigoplus_{i=1}^r\ker(L_0-\nu_i\id)\oplus\bigoplus_{j=1}^s\ker(L_0-\mu_j\id).\]
\end{proof}

\begin{Lemma}
\label{lemma2}
If there exists a nontrivial extension of $\SWA$-modules 
\[0\longrightarrow \Lxichi\longrightarrow M\longrightarrow \Lxichiprime\longrightarrow 0\] 
then $\chi$ and $\chi'$ coincide modulo $\Z$.
\end{Lemma}

\begin{proof}
	Suppose that there is a nontrivial extension
	\[0\longrightarrow L(\xi,\chi) \longrightarrow M\longrightarrow L(\xi',\chi')\longrightarrow 0.\]
	Set as in the previous proof, $L:=L(\xi,\chi)$ and $L':=L(\xi',\chi')$. 
	For $d\in\C$, let $M_d$ be the generalized $L_0$-eigenspace of $M$ attached to the eigenvalue $d$. 
	Set ${M[d]:=\bigoplus_{d'\in d+\Z}M_{d'}}$. It is a $\SWA$-submodule of $M$. 
	Then $M=\bigoplus_{\substack{d\in\C,\\0\leq\Re(d)<1}}M[d]$ is a direct sum decomposition of the $\SWA$-modules of $M$.
	For any $d$, the previous decomposition induces the following exact sequence
	\begin{equation}\label{sequencerest}0\longrightarrow L[d]\longrightarrow M[d]\longrightarrow L'[d]\longrightarrow 0.
	\end{equation}
	Assume $\chi-\chi'\notin\Z$. Since $L[d]=0$ if $d-\chi\notin\Z$, and $L[d']=0$ if $d'-\chi'\notin\Z$, 
	we get that  $M=M[\chi]\oplus M[\chi']$. 
	Taking $d=\chi$ and $d=\chi'$ in \eqref{sequencerest} we get
	\begin{align*}
	0\longrightarrow L[\chi]\longrightarrow &M[\chi]\longrightarrow 0,\\
	0\longrightarrow &M[\chi']\longrightarrow L'[\chi']\longrightarrow 0.
	\end{align*}
	Finally $M=L[\chi]\oplus L'[\chi']=L\oplus L'$ since $L$ and $L'$ are simple modules. So the sequence $0\rightarrow L\rightarrow M\rightarrow L'\rightarrow0$ splits, whence a contradiction.
\end{proof}

\begin{Proposition}
Suppose that either $k=-3+p/3$ with $(p,3)=1$, $p \geq 3$, or $k=-3+p/4$ with $(p,2)=1$, $p\geq 4$. 
Let $\WA\!\text{-}\!\Mod$ be the category of $\WA$-modules. Then 
\begin{equation*}
\Ext^1_{\WA\!\text{-}\!\Mod}(\Lxichiij{i,j}{s},\Lxichiij{i',j'}{s'})=0,
\end{equation*} 
where $\xi^{(s)}_{i,j}$ and $\chi^{(s)}_{i,j}$ are described in Proposition~\ref{2.7}.
\end{Proposition}

\begin{proof}
Assume $k=-3+p/3$  with $(p,3)=1$, $p \geq 3$. 
The argument for the coprincipal case is the same.	
It clearly appears that for all $1\leq i,i'\leq p-2$, $1\leq j\leq p-i-1$ and $1\leq j'\leq p-i'-1$, the differences $\chi^{(2)}_{i,j}-\chi^{(1)}_{i',j'}$ and $\chi^{(3)}_{i,j}-\chi^{(1)}_{i',j'}$ are not integers. 
According to Lemma \ref{lemma2}, any extension
\[0\longrightarrow \Lxichiij{i,j}{s}\longrightarrow M\longrightarrow \Lxichiij{i',j'}{s'}\longrightarrow 0,\] 
where exactly one of $\chi^{(s)}_{i,j}$ and $\chi^{(s')}_{i',j'}$ is with the first form is trivial. 
Applying $\psi$ we deduce that if $s\neq s'$ then 
\[\Ext^1_{\W^k(\g,f)\!\text{-}\!\Mod}(\Lxichiij{i,j}{s},\Lxichiij{i',j'}{s'})=0.\]
Suppose that $s=s'$. 
Using the $\psi$-action we can suppose that $s=s'=1$. 
According to \eqref{simple_mod}, since $\SWA$ is lisse, 
it suffices to show that there is not nontrivial extension
\begin{equation}\label{seq}
0\longrightarrow \mathbf{L}(E_{J_{\lambda_{i,j}-\frac{p}{3}x_0}}) \overset{\iota}{\longrightarrow} M\overset{\pi}{\longrightarrow} 
\mathbf{L}(E_{J_{\lambda_{i',j'}-\frac{p}{3}x_0}}) \longrightarrow 0.
\end{equation}
Set 
${L}_{i,j}:=	\mathbf{L}(E_{J_{\lambda_{i,j}-\frac{p}{3}x_0}})$ 
and 
${L}_{i',j'}:=\mathbf{L}(E_{J_{\lambda_{i',j'}-\frac{p}{3}x_0}})$. 
If $\chi^{(1)}_{i,j}=\chi^{(1)}_{i',j'}$, since the Zhu algebra $\Zhu(\SWA)$ is semisimple, the sequence
\begin{equation}
0\longrightarrow \top{({L}_{i,j})}\longrightarrow \top{M}\longrightarrow \top{({L}_{i',j'})}\longrightarrow 0.
\end{equation}
of $\Zhu(\SWA)$-modules is split. Applying the Zhu induction functor we get \eqref{seq} splits.
	
Let suppose $\chi^{(1)}_{i,j}>\chi^{(1)}_{i',j'}$. 
Set $M_{i',j'}:=\mathbf{M}(E_{J_{\lambda_{i',j'}-\frac{p}{3}x_0}})$.
Let $v_+$ be a primitive vector of $M_{i',j'}$ and $v \in M$ be such that $\pi(v)$ is the image of $v_+$ in $L_{i',j'}$. 
We have $\pi((L_0-\chi^{(1)}_{i',j'}\id)v)=(L_0-\chi^{(1)}_{i',j'}\id)\pi(v)=0$.
Hence $(L_0-\chi^{(1)}_{i',j'}\id)v\in\Ima\iota$. 
Since $\Ima\iota\simeq L_{i,j}$ and $\chi^{(1)}_{i,j}>\chi^{(1)}_{i',j'}$, $(L_0-\chi^{(1)}_{i',j'}\id)v\neq0$. 
As a consequence, it exists an injective $\WA$-modules homomorphism ${f:M_{i',j'}\to M}$ such that the diagram commutes:
	\begin{center}
		\begin{tikzcd}
			& M_{i',j'} \arrow[dashrightarrow]{dl}{f}\arrow{d}\\
			M \arrow[twoheadrightarrow]{r} & L_{i',j'}
		\end{tikzcd}
	\end{center}
Let us suppose  that the sequence \eqref{seq} does not split. 
The module $f(M_{i',j'})$ is a submodule of $M$. 
As a consequence, since $L_{i,j}\simeq\iota(L_{i,j})$ is simple, 
either $\iota(L_{i,j})\subset f(M_{i',j'})$ or $\iota(L_{i,j})\oplus f(M_{i',j'})$.
However if $\iota(L_{i,j})\oplus f(M_{i',j'})$ then the sequence
\[0\longrightarrow {L}_{i,j}\longrightarrow \iota(L_{i,j})\oplus f(M_{i',j'})\longrightarrow {L}_{i',j'}\longrightarrow 0\]
splits contradicting the fact that the sequence \eqref{seq} does not split. 
Hence $\iota(L_{i,j})\subset f(M_{i',j'})$.
Let $m\in M$. Since $\pi$ is surjective 
it exists $m_1\in M_{i',j'}$ such that $\pi(m)=\pi\circ f(m_1)$. Thus $m-f(m_1)\in\ker\pi$. 
We get $m\in\iota(L_{i,j})\subset f(M_{i',j'})$. 
Therefore $f$ is surjective.
It implies that $M$ is isomorphic to $M_{i',j'}$ as $\WA$-modules.
Hence ${[M_{i',j'}:L_{i,j}]\neq0}$. 
By \cite[Theorem 7.6]{AvE19} this happens only if it exists $\mu\in P_{0,+}$ such that ${[{\widehat{M}_k({\lambda}_{i',j'}-\frac{p}{3}x_0)}:\widehat{L}_k({\mu}-\frac{p}{3}x_0)]\neq0}$, 
where $\widehat{M}_k({\lambda}_{i',j'}-\frac{p}{3}x_0)$ is the Verma module of $\g$ with highest weight $\hat{\lam}_{i',j'}-\frac{p}{3}x_0$, 
and $E_{J_{\lambda_{i,j}-\frac{p}{3}x_0}}$ is a direct summand of $H_0^{\text{Lie}}(L(\mu-\frac{p}{3}x_0))$\footnote{The functor $M\mapsto H^{\text{Lie}}_0(M)$ defines a correspondence between a subcategory of the category $\mc{O}$ of $\g$-modules and the category of the finite dimensional representations of $U(\g,f)$, see \cite[Sect.~5]{Ara11} for a precise definition.}.

This first condition implies that $\mu\in W\circ\lambda_{i',j'}$ and we get $\lambda_{i,j}\in W\circ\mu$ from the second one. 
Hence $\lambda_{i,j}\in W\circ\lambda_{i',j'}$. 
Since $\hat{\lam}_{i,j}$ and $\hat{\lam}_{i',j'}$ are both dominant 
they are equal, and $\lambda_{i,j}=\lam_{i',j'}$ contradicting $\chi^{(1)}_{i,j}>\chi^{(1)}_{i',j'}$.
	
Finally if $\chi^{(1)}_{i,j}<\chi^{(1)}_{i',j'}$ by applying the duality functor to \eqref{seq} we are  back to the previous case $\chi^{(1)}_{i,j}>\chi^{(1)}_{i',j'}$.
\end{proof}

\begin{Example}\label{k-5/3}
	Let $k=-\frac{5}{3}$. 
	There exist nine simple $\W_{-\frac{5}{3}}(\g,f)$-modules. 
	We describe below the two orbits under the action of $\psi$:
	\begin{gather*}
	L(0,0)\overset{\psi}{\to} L\left(-\frac{1}{3},\frac{1}{6}\right)\overset{\psi}{\to} L\left(-\frac{2}{3},\frac{2}{3}\right)\overset{\psi}{\to} L\left(0,\frac{1}{2}\right)\overset{\psi}{\to}\\ L\left(-\frac{1}{3},\frac{2}{3}\right)\overset{\psi}{\to} L\left(\frac{1}{3},\frac{1}{6}\right)\overset{\psi}{\to}L(0,0),
	\end{gather*}
	\begin{equation*}
	L\left(-\frac{1}{2},\frac{7}{16}\right)\overset{\psi}{\to} L\left(\frac{1}{6},\frac{5}{48}\right)\overset{\psi}{\to} L\left(-\frac{1}{6},\frac{5}{48}\right)\overset{\psi}{\to} L\left(-\frac{1}{2},\frac{7}{16}\right).
	\end{equation*}
\end{Example}

A vertex algebra $V$ is said to be \emph{positive}~\cite{AvE19} if every irreducible $V$-modules besides $V$ itself has positive conformal dimension.
In our setting, $\W_k(\g,f)$ is positive if $\chi^{(s)}_{i,j}\geq0$ for all $s,i,j$ as in Proposition \ref{2.7}, and $k,\g,f$ as in this proposition. 
We observe the vertex algebras $W_{-\frac{5}{3}}(\g,f)$, $\W_{-\frac{4}{3}}(\g,f)$, 
$\W_{-\frac{7}{4}}(\g,f)$ and $\W_{-\frac{5}{4}}(\g,f)$ are positive. 

If $V$ is {\em unitary} (\cite[Sect. 2]{DL14}) 
then it is unitary as a module of the Virasoro subalgebra generated by the 
conformal vector as well. 
This forces the conformal dimension to be nonnegative. 
In particular, it is a positive vertex algebra. 
We expect the following result:

\begin{Conjecture}
	At level $k\in\{-\frac{5}{3},-\frac{4}{3},-\frac{7}{4},-\frac{5}{4}\}$, the vertex algebra 
	$\W_k(\sp{4},f_{subreg})$ is unitary. 
\end{Conjecture}

\begin{Remark}
For the other admissible levels as in the Main Theorem, which means for $k=-3+p/3$ with $\gcd{p}{3}=1$ and $p\geq6$, or $k=-3+p/4$ with $\gcd{p}{2}=1$ and $p\geq8$, the vertex algebra $\W_k(\sp{4},f_{subreg})$ is not positive because $\chi^{(1)}_{1,2}=-1+\frac{6}{p}$ and $\chi^{(1')}_{1,2}=-1+\frac{8}{p}$.
\end{Remark}

\section{Action of the component group of $\O_{subreg}$ over the simple $\SWA$-modules}\label{sec:componentgroup}

Realize $\sp{4}$ as in Sect.~\ref{sec:OPE}. The adjoint group of $\sp{4}$ is the quotient  $PSp_4=Sp_4/\{\pm I_4\}$\footnote{The matrix $I_n$ denotes the $n$-size square identity matrix.} where $Sp_4$ is the set of $4$-size square matrices $g$ such that $gJ_4g^T=J_4$. Fix an $\sl{2}$-triple $\{e,h,f\}$ such that $f\in\O_{subreg}$.
It is well-known that the component group $A(\O_{subreg})$ is isomorphic to $\Z/2\Z$ (see \cite[Corollary~6.1.6]{CM93}).
Let us describe explicitly the action of it on $\W^k(\sp{4},f_{subreg})$.

One may assume that
$$f=\begin{pmatrix}
	0&0\\-I_2&0
\end{pmatrix}\qquad\text{ and }\qquad
h=\begin{pmatrix}
	I_2&0\\0&-I_2
\end{pmatrix}.$$
Then the centralizer $\g^f$ is generated by the matrices $f$, $f_1$, $f_\theta$ and $h_2$ where
$$f_1=\begin{pmatrix}
	0&0&0&0\\0&0&0&0\\0&-1&0&0\\0&0&0&0
\end{pmatrix},\quad f_\theta=\begin{pmatrix}
0&0&0&0\\0&0&0&0\\0&0&0&0\\-1&0&0&0
\end{pmatrix},\quad\text{and}\quad h_2=\text{diag}(1,-1,1,-1).$$

Moreover let $G^\natural$ the stabilizer in $G$ of the $\sl{2}$-triple $\{e,h,f\}$. We get
$$G^\natural=\left\{\begin{pmatrix}
	A&0\\0&A
\end{pmatrix},\, AU_2A^T=U_2\right\}\simeq O(2).$$
Hence, $A(\O_{subreg})\simeq O(2)/SO(2)=\{I_4,U\}\simeq\Z/2\Z$ where $U=\begin{pmatrix}
	U_2&0\\0&U_2
\end{pmatrix}$.
Then $U$ acts on $\g^f$ by
$$U.f=f,\qquad U.h_2=-h_2,\qquad U.f_1=f_\theta\qquad\text{and}\qquad U.f_\theta=f_1.$$
It induces an automorphism $\Phi$ of $\WA$ defined by
\begin{align*}
	\Phi(L)=L,\quad\Phi(J)=-J\quad\text{and}\quad \Phi(G^\pm)=G^\mp.
\end{align*}
\begin{Remark}
	We can describe all automorphisms of $\WA$ using OPEs. Indeed an automorphism of $\WA$ is an isomorphism of conformal vertex algebra and preserves OPEs, the conformal vector and the conformal weights. Thus any automorphism of $\WA$ is of the form $\Phi^+_\alpha$ or $\Phi^-_\alpha$ ($\alpha\in\C^*$) where
	$$\Phi^\pm_\alpha(J)=\pm J,\quad \Phi^\pm_\alpha(G^+)=\alpha G^\pm\quad\text{and}\quad \Phi^\pm_\alpha(G^-)=\frac{1}{\alpha}G^\mp.$$
	With this notation we have $\Phi=\Phi^-_1$. 
\end{Remark}

\begin{Proposition}
	Suppose that $k$ is an admissible level.
	The automorphism $\Phi$ induces an action on the set of simple $\SWA$-modules:
	$$\Phi(\Lxichiij{i,j}{s})\simeq L(-(\xi_{i,j}^{(s)}+i-1),\chi_{i,j}^{(s)}).$$
	Moreover $\Phi$ is an involution (i.e. $\Phi^2=\id$) which preserves pairwise the simple modules of the first form.
	If $k$ is principal admissible then $\Phi$ induces a one-to-one correspondence between the modules of the second form and those of the third form.
	If $k$ is coprincipal admissible, $\Phi$ induces an action of the simple modules of the second form (remember that there is only two forms of simple modules in this case).
\end{Proposition}

\begin{proof}
	Let $\Lxichiij{i,j}{s}$ be a simple module of $\SWA$. Since $\Phi$ fixes $L$, the $L_0$-weights of $\Lxichiij{i,j}{s}$ remain unchanged and so the top-component of the simple module under the action of $\Phi(\WA)$ is still spanned by the vectors $(G_0^+)^m|\xi_{i,j}^{(s)},\chi_{i,j}^{(s)}\rangle$, $0\leq m\leq i-1$. 
	Moreover $$-J_0(G_0^+)^m|\xi_{i,j}^{(s)},\chi_{i,j}^{(s)}\rangle=-(\xi+m)(G_0^+)^m|\xi_{i,j}^{(s)},\chi_{i,j}^{(s)}\rangle$$ for all $m$. Then the smallest eigenvalue of $-J_0$ in the top-component is $-(\xi+i-1)$ associated with the vector $(G_0^+)^{i-1}|\xi_{i,j}^{(s)},\chi_{i,j}^{(s)}\rangle$. Hence $$\Phi(\Lxichiij{i,j}{s})\simeq L(-(\xi_{i,j}^{(s)}+i-1),\chi_{i,j}^{(s)}).$$
	
	If $s=1,1'$ then $\xi_{i,j}^{(s)}=\frac{1-i}{2}$ and $-(\xi_{i,j}^{(s)}+i-1)=\xi_{i,j}^{(s)}$. Thus $\Phi(\Lxichiij{i,j}{s})\simeq\Lxichiij{i,j}{s}$ for $s=1,1'$.
	Moreover suppose that $k$ is principal admissible then,
	$$\Phi(\Lxichiij{i,j}{2})\simeq\Lxichiij{i,p-i-j}{3}\quad\text{and}\quad\Phi(\Lxichiij{i,p-i-j}{3})\simeq\Lxichiij{i,j}{2},$$
	and if $k$ is coprincipal admissible,
	$$\Phi(\Lxichiij{i,j}{2'})\simeq\Lxichiij{i,p-2i-j}{2'}.$$
	In particular $\Phi$ is an involution.
\end{proof}

\begin{Remark}
	It seems that the exceptional $\W$-algebra $\W_k(\sp{4},f_{subreg})$ is the first known example of rational $\W$-algebra such that the corresponding component group acts non-trivially on the set of the simple $\SWA$-modules.
\end{Remark}

\begin{Example}
	Let $k=-\frac{5}{3}$. We described in the Example~\ref{k-5/3} the nine simple $\W_{-\frac{5}{3}}(\g,f)$-modules. There $\Phi$ fixes the modules $L(0,0)$, $L\left(0,\frac{1}{2}\right)$ and $L\left(-\frac{1}{2},\frac{7}{16}\right)$ and permutes the others as follows:
	\begin{align*}
		L\left(-\frac{1}{3},\frac{1}{6}\right)&\overset{\Phi}{\leftrightarrow}L\left(\frac{1}{3},\frac{1}{6}\right), \\
		L\left(-\frac{2}{3},\frac{2}{3}\right)&\overset{\Phi}{\leftrightarrow}L\left(-\frac{1}{3},\frac{2}{3}\right), \\
		L\left(\frac{1}{6},\frac{5}{48}\right)&\overset{\Phi}{\leftrightarrow} L\left(-\frac{1}{6},\frac{5}{48}\right).
	\end{align*}
\end{Example}

\section{Characters of admissible highest weight modules}\label{sec:characters}

Let $\g$ be any simple complex Lie algebra and $f$ any even nilpotent element of $\g$ as in the beginning of the previous section.
For any weight $\hat{\lam}=\lam+k\Lam_0\in\affh^*$, let $\ch(\widehat{L}_k(\lam))$ be the formal character of $\widehat{L}_k(\lam)$:
$$\ch(\widehat{L}_k(\lam))=\sum_{\mu\in\affh^*}\e^\mu\dim\widehat{L}_k(\lam)_\mu,$$
where $\widehat{L}_k(\lam)_\mu$ denotes the $\mu$-weight space.
It was proved by Kac and Wakimoto~\cite{KW88} that for $\lam\in Pr^k$,
$$\ch(\widehat{L}_k(\lam))=\frac{1}{R}\sum_{w\in\widehat{W}(\hat{\lam})}\epsilon(w)\e^{w\circ\hat{\lam}},$$
where $R$ is the Weyl denominator for $\affg$ and $w\circ\hat{\lam}=w(\hat{\lam}+\hat{\rho})-\hat{\rho}$.
This character can be considered as a meromorphic function of $(\tau,z)\in\mathcal{H}\times\h$, with $\mathcal{H}=\{\tau\in\C,\text{Im}\,\tau>0\}$, taking
$$\ch(\widehat{L}_k(\lam))(\tau,z)=\frac{1}{R}\sum_{w\in\widehat{W}(\hat{\lam})}\epsilon(w)\e^{-2\pi i\IP{w\circ\hat{\lam}}{\tau D-z}}.$$

Given a smooth $\affg$-module $M$ of level $k$, one can extend it to a $\mathcal{C}(\g,f,k)$-module $\mathcal{C}(M):=M\otimes \mathcal{F}(\g_{>0})$, which gives rise to a complex $(\mathcal{C}(M),d_{(0)})$. Its homology $H_f(M)$ is a $\WA$-module. 
As in \cite{KRW03}, let $\ch_{H_f(M)}$ be the Euler-Poincaré character of $H_f(M)$:
\begin{equation}\label{chHM}
	\ch_{H_f(M)}(q,h)=\sum_{j\in\Z}(-1)^j\ch_{H^j_f(M)}(q,z)=\sum_{j\in\Z}(-1)^j\tr_{H^j_f(M)}q^{L_0}\e^{2\pi i J_0^{\{h\}}},
\end{equation}
with $q=e^{2\pi i\tau}$ and $h\in\h^f$.

\begin{Theorem}[{\cite[Theorem~3.1]{KRW03}}]\label{ThmchHM}
	Let $M$ be the highest weight $\affg$-module with the highest weight $\hat{\lam}$ of level $k\neq h^\vee$, and suppose that $\ch(M)$ extends to a meromorphic function on $\mathcal{H}\times\h$ with at most simple poles at the hyperplane $T_{\alpha}=\{h\in\affh,\,\alpha(h)=0\}$, $\alpha\in\widehat{\Delta}^{re}$. Then
	
	\begin{align*}
		\ch(H_f(M))(q,h)=&\frac{q^{\frac{\IP{\hat{\lam}}{\hat{\lam}+2\hat{\rho}}}{2(k+h^\vee)}}}{\prod_{j=1}^{\infty}(1-q^j)^{\dim\h}}(R\ch(M)(\tau,h))\\
		&\times\prod_{n=1}^{\infty}\prod_{\alpha\in\Delta_{+}, \IP{\alpha}{x_0}=0}(1-q^{n-1}\e^{-2\pi i\IP{\alpha}{h}})^{-1}(1-q^n\e^{2\pi i\IP{\alpha}{h}})^{-1},
	\end{align*}
	where $h\in\h^f$.
\end{Theorem}

Until the end of the article, $\g$ refers to as the simple Lie algebra $\sp{4}$ and $f=f_{subreg}$. In this particular case $\h^f$ is one-dimensional generated by $\alpha_2^\vee$, viewed as an element of $\h$.

Following \cite{KRW03}, for any simple $\SWA$-modules $\Lxichiij{i,j}{s}$ described in Proposition~\ref{2.7}, one can associate the formal character
$$\ch_{\Lxichiij{i,j}{s}}(q,z):=\tr_{\Lxichiij{i,j}{s}}q^{L_0}z^{J_0}
=\sum_{d\in\chi^{(s)}_{i,j}+\Z_{\geq 0},\atop a\in\xi^{(s)}_{i,j}+\Z}\dim\Lxichiij{i,j}{s}_{a,d}q^{d}z^{a},$$
where $q=\e^{2\pi i\tau}$ and $z=e^{-2\pi i\nu}$, $\nu\in\C$.

Recall from the proof of Proposition~\ref{correspondence} that the simple modules $\Lxichiij{i,j}{1}$ (if $k$ is principal) and $\Lxichiij{i,j}{1'}$ (if $k$ is coprincipal) corresponds to some highest weight $\affg$-modules:
\begin{align}\label{isom_highest_weight}
	\Lxichiij{i,j}{s}\simeq H^0_f(\widehat{L}_k({\lambda}_{i,j}^{(s)})),
\end{align}
with $s=1,1'$ and $\lambda_{i,j}^{(s)}:=\lambda_{i,j}=(j-1)\varpi_1+(i-1)\varpi_2\in P_{0,+}$.

\begin{Proposition}
	\begin{enumerate}[font=\upshape,label=(\alph*)]
		\item Let $k=-3+p/3$ with $\gcd{p}{3}=1$, $p\geq3$ and $\xi^{(1)}_{i,j}$ and $\chi^{(1)}_{i,j}$ as in Proposition~\ref{2.7}\ref{2.7P} such that $\Lxichiij{i,j}{1}$ is a simple $\SWA$-module. Then
		\begin{align*}
			\ch_{\Lxichiij{i,j}{1}}(q,z)&=\frac{q^{\chi^{(1)}_{i,j}-\xi^{(1)}_{i,j}+j-1}}{\prod_{n\in\Z_{>0}}(1-q^n)^2}\prod_{n\in\Z_{>0}}(1-q^{n-1}z)^{-1}(1-q^nz^{-1})^{-1}\nonumber\\&\qquad\times\sum_{w\in\widehat{W}(k\Lambda_0)}\epsilon(w)q^{-(w(\hat{\lambda}_{i,j}+\hat{\rho})-\hat{\rho }|D+x_0)}z^{-(w(\hat{\lambda}_{i,j}+\hat{\rho})-\hat{\rho}|\frac{\alpha_2^\vee}{2})}.
		\end{align*}
		\item Let $k=-3+p/4$ with $\gcd{p}{2}=1$, $p\geq4$ and $\xi^{(1')}_{i,j}$ and $\chi^{(1')}_{i,j}$ as in Proposition~\ref{2.7}\ref{2.7CP} such that $\Lxichiij{i,j}{1'}$ is a simple $\SWA$-module. Then
		\begin{align*}
			\ch_{\Lxichiij{i,j}{1'}}(q,z)&=\frac{q^{\chi^{(1')}_{i,j}-\xi^{(1')}_{i,j}+j-1}}{\prod_{n\in\Z_{>0}}(1-q^n)^2}\prod_{n\in\Z_{>0}}(1-q^{n-1}z)^{-1}(1-q^nz^{-1})^{-1}\nonumber\\&\qquad\times\sum_{w\in\widehat{W}(k\Lambda_0)}\epsilon(w)q^{-(w(\hat{\lambda}_{i,j}+\hat{\rho})-\hat{\rho }|D+x_0)}z^{-(w(\hat{\lambda}_{i,j}+\hat{\rho})-\hat{\rho}|\frac{\alpha_2^\vee}{2})}.
		\end{align*}
	\end{enumerate}
\end{Proposition}

\begin{proof}
It is an immediate consequence of \cite[Theorem~5.5.2]{Ara11} that for $\lam\in P_{0,+}$, $H^i_f(\widehat{L}_k(\lam))=0$ for all $i\neq0$. 
Hence $\ch_{H_f(\widehat{L}_k({\lambda}_{i,j}^{(s)}))}=\ch_{H^0_f(\widehat{L}_k({\lambda}_{i,j}^{(s)}))}$.
According to the isomorphism \eqref{isom_highest_weight} and Theorem~\ref{ThmchHM}, we get that for $s=1,1'$, 
	\begin{align*}
	\ch_{\Lxichiij{i,j}{s}}(q,z)&=\ch_{H_f(\widehat{L}_k({\lambda}_{i,j}^{(s)}))}(q,-\nu \frac{\alpha_2^\vee}{2})\\
	&=\frac{q^{\frac{\IP{\hat{\lambda}_{i,j}}{\hat{\lambda}_{i,j}+2\rho}}{2(k+h^\vee)}}}{\prod_{n\in\Z_{>0}}(1-q^n)^2}\prod_{n\in\Z_{>0}}(1-q^{n-1}z)^{-1}(1-q^nz^{-1})^{-1}\widehat{R}\ch_{\widehat{L}_k({\lambda}_{i,j}^{(s)})}(q,z)\\	&=\frac{q^{\chi^{(s)}_{i,j}+\IP{x_0}{\hat{\lambda}_{i,j}}}}{\prod_{n\in\Z_{>0}}(1-q^n)^2}\prod_{n\in\Z_{>0}}(1-q^{n-1}z)^{-1}(1-q^nz^{-1})^{-1}\\
	&\qquad\times\sum_{w\in\widehat{W}(\hat{\lambda}_{i,j})}\epsilon(w)e^{2\pi i(w(\hat{\lambda}_{i,j}+\hat{\rho})-\hat{\rho}|-\tau D-\tau x_0-\nu\frac{\alpha_2^\vee}{2})}\\
	&=\frac{q^{\chi^{(s)}_{i,j}+\IP{x_0}{\hat{\lambda}_{i,j}}}}{\prod_{n\in\Z_{>0}}(1-q^n)^2}\prod_{n\in\Z_{>0}}(1-q^{n-1}z)^{-1}(1-q^nz^{-1})^{-1}\\
	&\qquad\times\sum_{w\in\widehat{W}(k\Lambda_0)}\epsilon(w)q^{-(w(\hat{\lambda}_{i,j}+\hat{\rho})-\hat{\rho}|D+x_0)}z^{-(w(\hat{\lambda}_{i,j}+\hat{\rho})-\hat{\rho}|\frac{\alpha_2^\vee}{2})}.
	\end{align*}
\end{proof}

Using the $\psi$-action, it is possible to get a similar expression for the other simple $\SWA$-modules. Indeed since $\psi$ is an isomorphism of vector spaces, it sends an eigenspace $(\Lxichiij{i,j}{s})_{d,a}$ to an eigenspace $(\psi(\Lxichiij{i,j}{s}))_{d',a'}$.

\begin{Lemma}\label{dim}
	Suppose that $k$ is an admissible level.
	Let $\xi^{(s)}_{i,j}$ and $\chi^{(s)}_{i,j}$ be such that $\Lxichiij{i,j}{s}$ is an irreducible $\SWA$-module as in Proposition~\ref{2.7}. Set $\Lxichiij{l,m}{r}:=\psi(\Lxichiij{i,j}{s})$. Then for $n\in\Z_{\geq0}$ and $a\in\Z$,
	\begin{align*}
		\dim(\Lxichiij{l,m}{r})_{\chi^{(r)}_{l,m}+n,\xi^{(r)}_{l,m}+a}=\dim(\Lxichiij{i,j}{s})_{\chi^{(s)}_{i,j}+n+a,\xi^{(s)}_{i,j}+a+i-1}.
	\end{align*}
\end{Lemma}

\begin{proof}
	Let $m\in(\Lxichiij{i,j}{s})_{\chi^{(s)}_{i,j}+n,\xi^{(s)}_{i,j}+a}$. Then
	\begin{align*}
		\psi(L_0)m&=(\chi^{(s)}_{i,j}-\xi^{(s)}_{i,j}+n-a+\frac{2+k}{2})m=(\chi^{(r)}_{l,m}+n-a+(i-1))m,\\
		\psi(J_0)m&=(\xi^{(s)}_{i,j}+a-(2+k))m=(\xi^{(r)}_{l,m}+a-(i-1))m.
	\end{align*}
	Hence $(\Lxichiij{i,j}{s})_{\chi^{(s)}_{i,j}+n,\xi^{(s)}_{i,j}+a}\simeq (\Lxichiij{l,m}{r})_{\chi^{(r)}_{l,m}+n-a+(i-1),\xi^{(r)}_{l,m}+a-(i-1)}$.
\end{proof}

\begin{Remark}
	By induction we get that for all $n\in\Z_{\geq0}, a\in\Z$, if $k=-3+p/3$, $\gcd{p}{3}=1$, $p\geq3$, then
	\begin{align*}
		(\Lxichiij{i,j}{s})_{\chi^{(s)}_{i,j}+n,\xi^{(s)}_{i,j}+a}&\simeq(\psi^6(\Lxichiij{i,j}{s}))_{\chi^{(s)}_{i,j}+n,\xi^{(s)}_{i,j}+a}\\
		&\simeq (\Lxichiij{i,j}{s}))_{\chi^{(s)}_{i,j}+6\xi^{(s)}_{i,j}+n+6a+18(2+k),\xi^{(s)}_{i,j}+a+6(2+k)},
	\end{align*}
	and if $k=-3+p/4$, $\gcd{p}{2}=1$, $p\geq4$, then
		\begin{align*}
		(\Lxichiij{i,j}{s})_{\chi^{(s)}_{i,j}+n,\xi^{(s)}_{i,j}+a}&\simeq(\psi^4(\Lxichiij{i,j}{s}))_{\chi^{(s)}_{i,j}+n,\xi^{(s)}_{i,j}+a}\\
		&\simeq (\Lxichiij{i,j}{s}))_{\chi^{(s)}_{i,j}+4\xi^{(s)}_{i,j}+n+4a+8(2+k),\xi^{(s)}_{i,j}+a+4(2+k)}.
	\end{align*}
\end{Remark}

\begin{Proposition}\label{psi_character}
	Suppose that $k$ is an admissible level. Let $\xi^{(s)}_{i,j}$ and $\chi^{(s)}_{i,j}$ be such that $\Lxichiij{i,j}{s}$ is a simple $\SWA$-module as in Proposition~\ref{2.7}.
	Then,
	\begin{equation*}
		\ch_{\psi(\Lxichiij{i,j}{s})}(q,z)=q^{\frac{2+k}{2}}z^{-(2+k)}\ch_{\Lxichiij{i,j}{s}}(q,q^{-1}z).
	\end{equation*}	
\end{Proposition}

\begin{proof}
	Denote $\psi(\chi^{(s)}_{i,j}):=\chi^{(s)}_{i,j}-\xi^{(s)}_{i,j}-(i-1)+\frac{2+k}{2}$ and $\psi(\xi^{(s)}_{i,j}):=\xi^{(s)}_{i,j}+(i-1)-(2+k)$.
	Using Lemma~\ref{dim} we get
	\begin{align*}
		\ch&_{\psi(\Lxichiij{i,j}{s})}(q,z)=\sum_{n\in\Z_{\geq0}\atop a\in\Z}\dim(\psi(\Lxichiij{i,j}{s}))_{\psi(\chi^{(s)}_{i,j})+n,\psi(\xi^{(s)}_{i,j})+a}q^{\psi(\chi^{(s)}_{i,j})+n}z^{\psi(\xi^{(s)}_{i,j})+a}\\
		&=\sum_{n\in\Z_{\geq0}\atop a\in\Z}\dim(\Lxichiij{i,j}{s})_{\chi^{(s)}_{i,j}+n+a,\xi^{(s)}_{i,j}+a+i-1}q^{\chi^{(s)}_{i,j}-\xi^{(s)}_{i,j}-(i-1)+\frac{2+k}{2}+n}z^{\xi^{(s)}_{i,j}+(i-1)-(2+k)+a}\\
		&=q^{\frac{2+k}{2}}z^{-(2+k)}\ch_{\Lxichiij{i,j}{s}}(q,q^{-1}z).
	\end{align*}
\end{proof}

\begin{Corollary}\label{chtwist}
	\begin{enumerate}[font=\upshape,label=(\alph*)]
		\item Let $k=-3+p/3$ with $\gcd{p}{3}=1$, $p\geq3$, $s=2,3$ and $\xi^{(s)}_{i,j}$ and $\chi^{(s)}_{i,j}$ as in Proposition~\ref{2.7}\ref{2.7P} such that $\Lxichiij{i,j}{s}$ is a simple $\SWA$-module. Then
		\begin{align*}
			&\ch_{\Lxichiij{i,j}{2}}(q,z)=\frac{q^{\chi^{(2)}_{i,j}+2p-2i-\frac{j+1}{2}}z^{-\frac{2p}{3}}}{\prod_{n\in\Z_{>0}}(1-q^n)^2}\prod_{n\in\Z_{>0}}(1-q^{n-1}z)^{-1}(1-q^nz^{-1})^{-1}\\
			&\hspace{3cm}\times\sum_{w\in\widehat{W}(k\Lambda_0)}\epsilon(w)q^{-(w(\hat{\lambda}_{j,p-i-j}+\hat{\rho})-\hat{\rho}|D+x_0-\alpha_2^\vee)}z^{-(w(\hat{\lambda}_{j,p-i-j}+\hat{\rho})-\hat{\rho}|\frac{\alpha_2^\vee}{2})}\\
			&=\frac{q^{\chi^{(2)}_{i,j}+2p-2i-\frac{j}{2}}z^{-\frac{2p}{3}+\frac{1}{2}}}{\prod_{n\in\Z_{>0}}(1-q^n)^2}\prod_{n\in\Z_{>0}}(1-q^{n-1}z)^{-1}(1-q^nz^{-1})^{-1}\\
			&\quad\times\sum_{w\in W\atop \eta\in Q^\vee}\epsilon(wt_{3\eta})q^{-(w(\lambda_{j,p-i-j}+\rho+p\eta)|x_0-\alpha_2^\vee)+3(\eta|\lambda_{j,p-i-j}+\rho)+\frac{3}{2}|\eta|^2p}\, z^{-(w(\lambda_{j,p-i-j}+\rho+p\eta)|\frac{\alpha_2^\vee}{2})},
		\end{align*}
		and
		\begin{align*}
			&\ch_{\Lxichiij{i,j}{3}}(q,z)=-\frac{q^{\chi^{(3)}_{i,j}+p-j-1}z^{-\frac{p}{3}}}{\prod_{n\in\Z_{>0}}(1-q^n)^2}\prod_{n\in\Z_{>0}}(1-q^{n-1}z)^{-1}(1-q^nz^{-1})^{-1}\nonumber\\
			&\hspace{3cm}\times\sum_{w\in\widehat{W}(k\Lambda_0)}\epsilon(w)q^{-(w(\hat{\lambda}_{p-i-j,i}+\hat{\rho})-\hat{\rho}|D+x_0-\frac{\alpha_2^\vee}{2})}z^{-(w(\hat{\lambda}_{p-i-j,i}+\hat{\rho})-\hat{\rho}|\frac{\alpha_2^\vee}{2})}\\
			&=-\frac{q^{\chi^{(3)}_{i,j}+p-j}z^{-\frac{p}{3}+\frac{1}{2}}}{\prod_{n\in\Z_{>0}}(1-q^n)^2}\prod_{n\in\Z_{>0}}(1-q^{n-1}z)^{-1}(1-q^nz^{-1})^{-1}\nonumber\\
			&\quad\times\sum_{w\in W\atop \eta\in Q^\vee}\epsilon(wt_{3\eta})q^{-(w(\lambda_{p-i-j,i}+\rho+p\eta)|x_0-\frac{\alpha_2^\vee}{2})+3(\eta|\lambda_{p-i-j,i}+\rho)+\frac{3}{2}|\eta|^2p}\, z^{-(w(\lambda_{p-i-j,i}+\rho+p\eta)|\frac{\alpha_2^\vee}{2})}.
		\end{align*}
		\item Let $k=-3+p/4$ with $\gcd{p}{2}=1$, $p\geq4$ and $\xi^{(2')}_{i,j}$ and $\chi^{(2')}_{i,j}$ as in Proposition~\ref{2.7}\ref{2.7CP} such that $\Lxichiij{i,j}{2'}$ is a simple $\SWA$-module. Then
		\begin{align*}
			&\ch_{\Lxichiij{i,j}{2'}}(q,z)=-\frac{q^{\chi^{(2')}_{i,j}+i+j-2}z^{\frac{p}{4}}}{\prod_{n\in\Z_{>0}}(1-q^n)^2}\prod_{n\in\Z_{>0}}(1-q^{n-1}z)^{-1}(1-q^nz^{-1})^{-1}\nonumber\\
			&\hspace{3cm}\times\sum_{w\in\widehat{W}(k\Lambda_0)}\epsilon(w)q^{-(w(\hat{\lambda}_{j,i}+\hat{\rho})-\hat{\rho}|D+x_0+\frac{\alpha_2^\vee}{2})}z^{-(w(\hat{\lambda}_{j,i}+\hat{\rho})-\hat{\rho}|\frac{\alpha_2^\vee}{2})}\\
			&=-\frac{q^{\chi^{(2')}_{i,j}+i+j}z^{\frac{p}{4}+\frac{1}{2}}}{\prod_{n\in\Z_{>0}}(1-q^n)^2}\prod_{n\in\Z_{>0}}(1-q^{n-1}z)^{-1}(1-q^nz^{-1})^{-1}\nonumber\\
			&\quad\times\sum_{w\in W\atop \eta\in Q^\vee}\epsilon(wt_{4\eta})q^{-(w(\lambda_{j,i}+\rho+p\eta)|x_0+\frac{\alpha_2^\vee}{2})+4(\eta|\lambda_{j,i}+\rho)+2|\eta|^2p}\, z^{-(w(\lambda_{j,i}+\rho+p\eta)|\frac{\alpha_2^\vee}{2})}.
		\end{align*}
	\end{enumerate}
\end{Corollary}

\begin{proof}
	If $k$ is principal admissible then the irreducible module $\Lxichiij{i,j}{3}$ is isomorphic to $\psi(\Lxichiij{p-i-j,i}{1})$ whereas $\Lxichiij{i,j}{2}$ is isomorphic to $\psi^2(\Lxichiij{j,p-i-j}{1})$. Hence
	\begin{align*}
		\ch_{\Lxichiij{i,j}{3}}(q,z)&=q^{\frac{2+k}{2}}z^{-(2+k)}\ch_{\Lxichiij{p-i-j,i}{1}}(q,q^{-1}z),
	\end{align*}
	and
	\begin{align*}
		\ch_{\Lxichiij{i,j}{2}}(q,z)&=q^{\frac{2+k}{2}}z^{-(2+k)}\ch_{\Lxichiij{p-i-j,i}{3}}(q,q^{-1}z)\\
		&=q^{2(2+k)}z^{-2(2+k)}\ch_{\Lxichiij{j,p-i-j}{1}}(q,q^{-2}z).
	\end{align*}
	If $k$ is coprincipal admissible then the irreducible module $\Lxichiij{j,i}{1'}$ is isomorphic to $\psi(\Lxichiij{i,j}{2'})$. Hence
	\begin{align*}
		\ch_{\Lxichiij{i,j}{2'}}(q,z)=q^{\frac{2+k}{2}}z^{(2+k)}\ch_{\Lxichiij{j,i}{1'}}(q,qz).
	\end{align*}
	Moreover if $A(q,z)$ denotes the product
	$$\frac{\prod_{n\in\Z_{>0}}(1-q^{n-1}z)^{-1}(1-q^nz^{-1})^{-1}}{\prod_{n\in\Z_{>0}}(1-q^n)^2},$$
	then $A(q,q^{-1}z)=-qz^{-1}A(q,z)$, $A(q,q^{-2}z)=q^3z^{-2}A(q,z)$ and $A(q,qz)=-zA(q,z)$. This remark allows to complete the computation.
\end{proof}

We have already showed that the simple $\W_k(\sp{4},f_{subreg})$-modules with the form $\Lxichiij{i,j}{1}$ (if $k$ is principal admissible) and $\Lxichiij{i,j}{1'}$ (if $k$ is coprincipal admissible) are isomorphic to the zero-th homology of the $H_f$-reduction of certain highest weight $\affg$-modules. 
We expect that this result holds for all simple $\W_k(\sp{4},f_{subreg})$-modules. More precisely we expect that for any simple $\SWA$-module $\Lxichiij{i,j}{s}$, it exists an admissible weight $\lam^{(s)}_{i,j}\in\affh^*$ such that
\begin{equation*}
	H^l_f(\widehat{L}_k(\lam^{(s)}_{i,j}))=
	\left\{\begin{aligned}
		&\Lxichiij{i,j}{s}&&\qquad\text{if } l=0,\\
		&0&&\qquad\text{else}.
	\end{aligned}\right.
\end{equation*}
If such a weight exists then it satisfies the following equation:
\begin{equation}\label{eqweight}
	\left\{
	\begin{aligned}
		\xi^{(s)}_{i,j}&=\Lambda^{(s)}_{i,j}(-\frac{\alpha_2^\vee}{2})~\mod\Z,\\
		\chi^{(s)}_{i,j}&=\frac{(\Lambda^{(s)}_{i,j}|\Lambda^{(s)}_{i,j}+2\rho)}{2(k+h^\vee)}-\Lambda^{(s)}_{i,j}(x_0)~\mod\Z.
	\end{aligned}\right.
\end{equation}
Solving this system for all pairs $(\xi^{(s)}_{i,j},\chi^{(s)}_{i,j})$, we get candidates for the weights $\lam_{i,j}^{(s)}$ which answer to the conjecture. Among all of them we keep only the admissible ones.

For instance, for all $1\leq i\leq p-1$ and $1\leq j\leq p-i-1$, $\lam^{(2)}_{i,j}\in\{\Lambda^{(2)+}_{i,j}, \Lambda^{(2)-}_{i,j}\}$
where $\Lambda^{(2)\pm}_{i,j}=\frac{-6-6i-3j+4p\pm(3j-2p)}{6}\varpi_1+\frac{-3+6i+3j-2p}{3}\varpi_2$.
However the weight $\Lambda^{(2)-}_{i,j}$ is not regular dominant whereas $\Lambda^{(2)+}_{i,j}$ is an admissible weight. As a consequence we set $$\lam^{(2)}_{i,j}=\Lambda^{(2)+}_{i,j}=(-i+p/3-1)\varpi_1+(2i+j-2p/3-1)\varpi_2.$$
Similarly we set
\begin{align*}
	{\lambda}^{(3)}_{i,j}&=(-i+p/3-1)\varpi_1+(i+j-p/3-1)\varpi_2,\\
	{\lambda}^{(2')}_{i,j}&=(-i+p/4-1)\varpi_1+(2i+j-p/2-1)\varpi_2.
\end{align*}

\begin{Proposition}\label{weights}
	\begin{enumerate}[font=\upshape,label=(\alph*)]
		\item Let $k=-3+p/3$ with $\gcd{p}{3}=1$, $p\geq3$, $s=2,3$ and $\xi^{(s)}_{i,j}$ and $\chi^{(s)}_{i,j}$ as in Proposition~\ref{2.7}\ref{2.7P} such that $\Lxichiij{i,j}{s}$ is a simple $\SWA$-module. Then
		\begin{equation}\label{ch2}
			\ch_{H_f(\widehat{L}_k({\lambda}^{(2)}_{i,j}))}(q,z)=\ch_{\Lxichiij{i,j}{2}}(q,z),
		\end{equation}
		and
		\begin{equation}\label{ch3}
			\ch_{H_f(\widehat{L}_k({\lambda}^{(3)}_{i,j}))}(q,z)=\ch_{\Lxichiij{i,j}{3}}(q,z).
		\end{equation}
		\item Let $k=-3+p/4$ with $\gcd{p}{2}=1$, $p\geq4$ and $\xi^{(2')}_{i,j}$ and $\chi^{(2')}_{i,j}$ as in Proposition~\ref{2.7}\ref{2.7P} such that $\Lxichiij{i,j}{2'}$ is a simple $\SWA$-module. Then
		\begin{equation}\label{ch2'}
			\ch_{H_f(\widehat{L}_k({\lambda}^{(2')}_{i,j}))}(q,z)=\ch_{\Lxichiij{i,j}{2'}}(q,z).
		\end{equation}
	\end{enumerate}
\end{Proposition}

\begin{proof}
	We only detail the proof of \eqref{ch2}. The formula \eqref{ch3} and \eqref{ch2'} are obtained with a similar computation.
	By \cite[Section~3]{KRW03},
	\begin{align*}
		\ch_{H_f(\widehat{L}_k(\lambda^{(2)}_{i,j}))}(q,z)&=\frac{q^{\chi^{(2)}_{i,j}+\frac{j-3}{2}}}{\prod_{n\in\Z_{>0}}(1-q^n)^2}\prod_{n\in\Z_{>0}}(1-q^{n-1}z)^{-1}(1-q^nz^{-1})^{-1}\nonumber\\
		&\qquad\times\sum_{w\in\widehat{W}(\hat{\lambda}^{(2)}_{i,j})}\epsilon(w)q^{-(w(\hat{\lambda}^{(2)}_{i,j}+\hat{\rho})-\hat{\rho}|D+x_0)}z^{-(w(\hat{\lambda}^{(2)}_{i,j}+\hat{\rho})-\hat{\rho}|\frac{\alpha_2^\vee}{2})}.
	\end{align*}
	Moreover $\widehat{\Delta}(\hat{\lambda}^{(2)}_{i,j})=y\left(\widehat{\Delta}(k\Lambda_0)\right)$ where $y=-r_{\alpha_1}\!r_{\alpha_2}\!t_{-\varpi_1^\vee}\in\widetilde{W}$ and $\widehat{W}(\hat{\lambda}^{(2)}_{i,j})=\{ywy^{-1},\,w\in\widehat{W}(k\Lam_0)\}$. Hence
	\begin{align*}
		&\ch_{H_f(\widehat{L}_k(\lambda^{(2)}_{i,j}))}(q,z)=\frac{q^{\chi^{(2)}_{i,j}+\frac{j-3}{2}}}{\prod_{n\in\Z_{>0}}(1-q^n)^2}\prod_{n\in\Z_{>0}}(1-q^{n-1}z)^{-1}(1-q^nz^{-1})^{-1}\nonumber\\
		&\qquad\times\sum_{w\in\widehat{W}(k\Lambda_0)}\epsilon(w)q^{-(y_2wy_2^{-1}(\hat{\lambda}^{(2)}_{i,j}+\hat{\rho})-\hat{\rho}|D+x_0)}z^{-(y_2wy_2^{-1}(\hat{\lambda}^{(2)}_{i,j}+\hat{\rho})-\hat{\rho}|\frac{\alpha_2^\vee}{2})}\\
		&=\frac{q^{\chi^{(2)}_{i,j}+2p-2i-\frac{j}{2}}z^{-\frac{2p}{3}+\frac{1}{2}}}{\prod_{n\in\Z_{>0}}(1-q^n)^2}\prod_{n\in\Z_{>0}}(1-q^{n-1}z)^{-1}(1-q^nz^{-1})^{-1}\\
		&\qquad\times\sum_{w\in W\atop \eta\in Q^\vee}\epsilon(wt_{3\eta})q^{((r_{\alpha_1}r_{\alpha_2}-2\id)w(-(\lambda_{j,p-i-j}+\rho)+p\eta)|\varpi_1)-3(\eta|\lambda_{j,p-i-j}+\rho)+\frac{3}{2}|\eta|^2p}\\
		&\hspace{2cm}\times z^{(r_{\alpha_1}r_{\alpha_2}w(-(\lambda_{j,p-i-j}+\rho)+p\eta)||\frac{\alpha_2^\vee}{2})}\\
		&=\frac{q^{\chi^{(2)}_{i,j}+2p-2i-\frac{j}{2}}z^{-\frac{2p}{3}+\frac{1}{2}}}{\prod_{n\in\Z_{>0}}(1-q^n)^2}\prod_{n\in\Z_{>0}}(1-q^{n-1}z)^{-1}(1-q^nz^{-1})^{-1}\\
		&\qquad\times\sum_{w\in W\atop \eta\in Q^\vee}\epsilon(wt_{-3\eta})q^{-((\id+2r_{\alpha_1}r_{\alpha_2})w(\lambda_{j,p-i-j}+\rho+p\eta)|\varpi_1)+3(\eta|\lambda_{j,p-i-j}+\rho)+\frac{3}{2}|\eta|^2p}\\
		&\hspace{2cm}\times z^{-(w(\lambda_{j,p-i-j}+\rho+p\eta)||\frac{\alpha_2^\vee}{2})}.
	\end{align*}
	Moreover for all $w\in W$, 
	$(-r_{\alpha_1}r_{\alpha_2}w(\lambda_{j,p-i-j}+\rho+p\eta)|\varpi_1)=(w(\lambda_{j,p-i-j}+\rho+p\eta)|\frac{\alpha_2^\vee}{2}).$
	As a consequence,
	\begin{align*}
		&\ch_{H_f(\widehat{L}_k(\lambda^{(2)}_{i,j}))}(q,z)=\frac{q^{\chi^{(2)}_{i,j}+2p-2i+\frac{j}{2}}z^{-\frac{2p}{3}+\frac{1}{2}}}{\prod_{n\in\Z_{>0}}(1-q^n)^2}\prod_{n\in\Z_{>0}}(1-q^{n-1}z)^{-1}(1-q^nz^{-1})^{-1}\\
		&\;\times\sum_{w\in W\atop \eta\in Q^\vee}\epsilon(wt_{3\eta})q^{-(w(\lambda_{j,p-i-j}+\rho+p\eta)|x_0-\alpha_2^\vee)+3(\eta|\lambda_{j,p-i-j}+\rho)+\frac{3|\eta|^2}{2}p}\,z^{-(w(\lambda_{j,p-i-j}+\rho+p\eta)|\frac{\alpha_2^\vee}{2})}\\
		&=\ch_{\Lxichiij{i,j}{2}}(q,z).
	\end{align*}
\end{proof}

We are in a position to formulate a conjecture.

\begin{Conjecture}\label{conjecture}
	Let $k$ be an admissible level. Then, all the simple $\SWA$-modules $\Lxichiij{i,j}{s}$ as in Proposition~\ref{2.7} are obtained from the Drinfeld-Sokolov reduction of a highest weight $\widehat{\g}$-module. More precisely,
	\begin{enumerate}[font=\upshape,label=(\alph*)]
		\item If $k=-3+p/3$ with $\gcd{p}{3}=1$, $p\geq3$, $\xi^{(s)}_{i,j}$ and $\chi^{(s)}_{i,j}$ as in Proposition~\ref{2.7}\ref{2.7P}, then 
		$$H^l_f(\widehat{L}_k(\lam^{(s)}_{i,j}))=
		\left\{\begin{aligned}
			&\Lxichiij{i,j}{s}&&\qquad\text{if } l=0,\\
			&0&&\qquad\text{else}.
		\end{aligned}\right.$$
		\item If $k=-3+p/4$ with $\gcd{p}{2}=1$, $p\geq4$, $\xi^{(s)}_{i,j}$ and $\chi^{(s)}_{i,j}$ as in Proposition~\ref{2.7}\ref{2.7CP}, then 
		$$H^l_f(\widehat{L}_k(\lam^{(s)}_{i,j}))=
		\left\{\begin{aligned}
			&\Lxichiij{i,j}{s}&&\qquad\text{if } l=0,\\
			&0&&\qquad\text{else}.
		\end{aligned}\right.$$
	\end{enumerate}
\end{Conjecture}

\newcommand{\etalchar}[1]{$^{#1}$}

\end{document}